\documentclass[12pt,a4paper]{amsart}
\usepackage{amssymb,cmap,verbatim}
\usepackage{stmaryrd,hyperref} % to get the proper \bigtriangleup operator for diagonal intersection
\usepackage{mathrsfs}

\newtheorem{thm}{Theorem}[section]
\newtheorem{lemma}[thm]{Lemma}
\newtheorem{cor}[thm]{Corollary}
\newtheorem{claim}{Claim}[thm]
\newtheorem{example}[thm]{Example}

\newtheorem{prop}[thm]{Proposition}
\newtheorem{fact}[thm]{Fact}
\newtheorem{question}[thm]{Question}

\theoremstyle{definition}
\newtheorem{defn}[thm]{Definition}

\theoremstyle{remark}
\newtheorem{remark}[thm]{Remark}

\DeclareMathOperator{\cf}{cf}

\DeclareMathOperator{\dom}{dom}
\DeclareMathOperator{\rng}{rng}

\renewcommand{\mid}{\mathrel{|}\allowbreak}
\renewcommand{\restriction}{\mathbin\upharpoonright}

\newenvironment{cproof}{\paragraph{\emph{Proof.\,}}}{\hfill{$\boxminus$}\par\vspace{2mm}}

\def\borelstar{Borel\nobreak\hspace{1pt}$^*$}
\def\Orb{\mbox{\rm Orb}}

\title{On Borel subsets of Generalized Baire Spaces}

\author{Tapani Hyttinen}
\address{Department of Mathematics and Statistics, University of Helsinki, Helsinki 00560, Finland.}

\author{Miguel Moreno}
\address{Department of Mathematics and Statistics, University of Helsinki, Helsinki 00560, Finland.}

\author{Jouko Väänänen}
\address{Department of Mathematics and Statistics, University of Helsinki, Helsinki 00560, Finland.}

\begin{document}
\begin{abstract}
We develop Descriptive Set Theory in Generalized Baire Spaces without assuming $\kappa^{<\kappa}=\kappa$. We point out that without this assumption the basic topological concepts of these spaces have to be slightly modified in order to obtain a meaningful theory. This modification has no effect if $\kappa^{<\kappa}=\kappa$. After developing the basic theory we apply it to the question whether the orbits of models  of a fixed cardinality $\kappa$ in the space  $\kappa^\kappa$ are $\kappa$-Borel in our generalized sense. It turns out that this question depends, as is the case when $\kappa^{<\kappa}=\kappa$, on stability theoretic properties (structure vs. non-structure) of the first order theory of the model.
\end{abstract}

\date{\today}

\maketitle

\section{Introduction}

We develop Descriptive Set Theory in Generalized Baire Spaces without cardinality arithmetic assumptions, such as the Continuum Hypothesis, or the assumption $\kappa=\kappa^{<\kappa}$ for the relevant $\kappa$. In our main results, we use a variant of the usual product topology rather than the so-called bounded topology, as we think it is more natural for our results in the case $\kappa<\kappa^{<\kappa}$, which is our main concern. Both concepts will be defined and reviewed below, and we will give reasons for using the former. We also observe that in order to develop Descriptive Set Theory in a meaningful way, we have to focus on unions of at most $\kappa$ basic open sets, for the relevant $\kappa$, rather than on unions of \emph{arbitrary} families of basic open sets. Otherwise there would be simply too many  open sets. Of course, this does not make any difference if we assume $\kappa^{<\kappa}=\kappa$, or if we use the product topology. 
%Moreover, we observe that in order for the concept of $\kappa$-open to cover some sets important for our applications, we have to modify slightly the received concept of basic open set. Again, the modification makes no difference if $\kappa^{<\kappa}=\kappa$.

Our applications come from model theory. Indeed, model theory is a natural framework for Descriptive Set Theory in Generalized Baire Spaces. We cannot rely on natural sciences as  a guide and a source of inspiration in the same way as we  do in the classical Baire space $\omega^\omega$. The idea of a measurement with an $\omega_1$-sequence of successive improvements  does not arise in natural sciences in any apparent way. On the other hand, uncountable models are essential in the study of first order, as well as infinitary theories, as e.g. the Classification Theory of Shelah \cite{MR1083551} clearly demonstrates. 

It is natural to ask e.g. whether the orbit (i.e. isomorphism type) of a model is in some natural sense Borel, rather than merely analytic. Such investigation has been already done under cardinality arithmetic assumptions (\cite{MR1111753}, \cite{MR3235820},\cite{MR3724375}, \cite{MR4194559}, \cite{MR4855463}). The purpose of this paper is try to prove these results without extra assumptions.

Section 2 of the paper introduces and reviews the basic topological concepts that we need. In particular, we indicate the difficulties we run into with the bounded topology when $\kappa<\kappa^{<\kappa}$, and what our solution for avoiding them is. Section 3 is devoted to the analogue of \borelstar\ in the case $\kappa<\kappa^{<\kappa}$. Here \borelstar\ refers to the definition of Borel sets via Borel codes. Section 3 also introduces the concept of $\kappa$-$\Sigma^1_1$-definability and establishes its main properties. Finally, Section 3 establishes a Hierarchy Theorem for our concept of $\kappa$-Borelness.

In section 4 we prove, that if $T$ is a countable non-classifiable first order theory and $\kappa>\omega$ is regular, then the orbits of some models of $T$ of cardinality $\kappa$ are not $\kappa$-Borel and the isomorphism relation of models of $T$ of cardinality $\kappa$ is not $\kappa$-Borel. Previous results to the same effect (e.g. \cite{MR3235820}) used the assumption $\kappa^{<\kappa}=\kappa$. Of course, this is in sharp contrast to the fact that the orbits of all countable models are Borel \cite{MR200133}. 

In Section 5 we consider a regular $\kappa$ with $\kappa^\omega=\kappa$. Suppose $T$ is a countable $\omega$-stable NDOP shallow first order theory. We show that the orbits of all models of $T$ of cardinality $\kappa$ are $\kappa$-Borel. Moreover, if $T$ has only at most $\kappa$ non-isomorphic models of cardinality $\kappa$, then  the isomorphism relation of models of $T$ of cardinality $\kappa$ is  $\kappa$-Borel. Again,  previous results to the same effect (e.g. \cite{MR3235820}) use the assumption $\kappa^{<\kappa}=\kappa$. The results of Section 5 do not seem at all optimal to us, so there is probably  room for improvement.

This project has received funding from the Academy of Finland (decision number 368671) and the European Research Council (ERC) under the
European Union’s Horizon 2020 research and innovation programme (grant agreement No
101020762).

\section{Basic  topology}

%Suppose $\kappa$ is a regular cardinal. 
In generalized Descriptive Set Theory the set $\kappa^\kappa$ is usually, and we take here the same starting point,  endowed with the so-called \emph{bounded} topology generated by the basic open sets 
\begin{equation}\label{old}
N_\eta=\{\zeta\in \kappa^\kappa\mid\eta\subseteq \zeta\},
\end{equation} where $\eta:\alpha\rightarrow \kappa$ and  $\alpha<\kappa$   (\cite{MR1110032,MR1242054,MR3235820,MR4409720}). This is in analogy with the classical Baire  space. In contrast, the most common topology of a product space in general topology is the so-called product  topology, a.k.a. Tychonov-topology. A natural notion of ``$\kappa$-openness", derived from  the product topology, plays a major role below.

Continuing the review of the bounded topology (\ref{old}), a set $U$ is  defined as open,\footnote{Whenever we say in this paper that a set is open or closed, we mean open or closed in the bounded topology.}   if it is the union of a family of basic open sets:
\begin{equation}\label{old1}
U=\bigcup_{i\in I}N_{\eta_i}.
\end{equation} The bounded topology of $2^\kappa$ is defined similarly. This definition raises the question, how many basic open sets are there after all, especially as in the classical Baire space $\omega^\omega$ and Cantor space $2^\omega$ there are only countably many basic open sets. Obviously, the number of basic open sets (\ref{old}) is $\kappa^{<\kappa}$. This suggests that the Descriptive Set Theory of $\kappa^\kappa$  is quite different according to whether or not $\kappa^{<\kappa}=\kappa$. 

The  spaces built on (\ref{old}) and (\ref{old1}) have given rise to a lot of research, starting with \cite{MR1110032} and \cite{MR1242054} among the early ones and e.g. 
\cite{MR3235820} and \cite{MR4409720} among the more recent ones.  In this pioneering research the assumption $\kappa^{<\kappa}=\kappa$ is made virtually without exception. If this assumption is dropped, a problem with  the topology based on the definition (\ref{old1}) arises:

\begin{lemma}\label{dead-end}
    Suppose $2^\lambda=2^\kappa$, $\kappa>\lambda$. Then for all $A\subseteq \kappa^\kappa$ there is a closed $C\subseteq \kappa^\kappa\times 2^\kappa$ such that $A$ is the first projection  of $C$ i.e. $$A=\{f\in \kappa^\kappa:\exists g\in 2^\kappa((f,g)\in C)\}.$$ 
%In particular, $A$ is \borelstar (in the sense of  \cite{MR1242054}).
\end{lemma}

\begin{cproof}
%(apu178) 
Let $\pi:\kappa^\kappa\to 2^{\lambda}$ be a bijection and
$$C=\{(\eta,\xi)\in \kappa^\kappa\times 2^\kappa: \eta=\pi^{-1}(\xi\restriction\lambda)\in A\}.$$Clearly, the first projection of $C$ is $A$. To prove that $C$ is closed, suppose $(\eta,\xi)\notin C$. If it happens that $\pi^{-1}(\xi\restriction\lambda)\notin A$, then we let
$$U=\{(\eta',\xi'):\xi'\restriction\lambda=\xi\restriction\lambda, \eta'\in \kappa^\kappa\}.$$ The set $U$ is open, $(\eta,\xi)\in U$ and $C\cap U=\emptyset$. On the other hand, if it happens that  $\pi^{-1}(\xi\restriction\lambda)\in A$, then there is $\alpha<\kappa$ such that $\eta(\alpha)\ne\pi^{-1}(\xi\restriction\lambda)(\alpha)$ and we let
$$U=\{(\eta',\xi'):\xi'\restriction\lambda=\xi\restriction\lambda, \eta'(\alpha)=\eta(\alpha)\}.$$ The set $U$ is again open, $(\eta,\xi)\in U$ and $C\cap U=\emptyset$. 

%To prove that $A$ is \borelstar, consider the \borelstar code $(T,L)$, where   $T$ is  $2^{<\omega}$ as a tree and the labelling $L$ is defined as follows: Each node of $T$ is labelled with $\wedge$. Suppose $\xi$ is a maximal branch in $T$ i.e. $\xi\in 2^\omega$. Let $\eta=\pi^{-1}(\xi)$. If $\eta\in A$,  we let $L(\xi)=\{\eta\}$. If $\eta\notin A$, we let $L(\xi)=\emptyset$. Now $\eta\in A$ if and only if I has a winning strategy in $G_{T,L}(\eta)$.
\end{cproof}

This lemma shows that we cannot hope for any meaningful theory of analytic sets in the case that $2^\omega=2^\kappa$, $\kappa>\omega$. Still e.g. orbits of  structures of size $\kappa$ are canonical examples of analytic sets in these spaces. 
%Also the concept of \borelstar becomes empty. This leads us to the following slight modification of the basic topology of $\kappa^\kappa$ or $2^\kappa$:
%
A natural way out of the \emph{cul-de-sac} of Lemma~\ref{dead-end} is to define a set to be ``$\kappa$-open" if (\ref{old1}) is satisfied with $|I|\le\kappa$. Indeed, this kills the proof of Lemma~\ref{dead-end} in the sense that the set $C$ therein is not ``$\kappa$-closed", i.e. the complement of a ``$\kappa$-open" set. However, it turns out that this concept of ``$\kappa$-open" is too restrictive, as the following example shows:

\begin{example}\label{example11}
Suppose $2^\lambda>\kappa>\lambda$. The set $\{f\in\kappa^\kappa:f(\lambda)=f(\lambda+1)\}$ is clopen but not a union of $\le\kappa$ sets of the form (\ref{old}).
\end{example}

This example shows that if we define  a set $A$ to be ``$\kappa$-open" if (\ref{old1}) is satisfies with $|I|\le\kappa$, the extremely simple and ``basic" set of  Example~\ref{example11}, need not be ``$\kappa$-open". Still such sets are the basis of definability theory in infinitary model theory. 

The following is further example of the difficulties one runs into when developing a theory of Borel sets on the basis of bounded topology:

\begin{example}\label{example1}
We say that $X\subseteq \kappa^\kappa$ is $\kappa$-Borel\hspace{2pt}$^+$ if it is in the closure of basic open sets (\ref{old}) under complements and unions of size at most $\kappa$. If $2^\omega>\kappa>\omega$, then there are some very simple clopen sets, e.g. $\{\eta\in\kappa^\kappa\mid 
\eta(\omega)>0\}$, that are not $\kappa$-Borel\hspace{2pt}$^+$\!. 
To see this, let $\circledast(Y,\eta)$ be the condition $$``N_{\eta^\frown 0}\cup N_{\eta^\frown 1}\subseteq Y\text{ or }N_{\eta^\frown 0}\cup N_{\eta^\frown 1}\subseteq \kappa^\kappa\backslash Y",$$ where $Y\subseteq\kappa^\kappa$ and $\eta\in \kappa^\omega$.
We say that $Y\subseteq \kappa^\kappa$ is \emph{bad} if for almost all (i.e. for all but at most $\kappa$ many) $\eta\in\kappa^\omega$, $\circledast(Y,\eta)$ holds.
We show that all $\kappa$-Borel\hspace{2pt}$^+$ sets are bad.
    Let us start by showing that basic open sets $N_\eta$, where $\eta:\alpha\rightarrow \kappa$, are bad sets. Let $\alpha<\kappa$ and $\eta:\alpha\rightarrow \kappa$. 
  
         {\bf Case $\alpha\leq\omega$.} Clearly, for all $\xi:\omega\rightarrow\kappa$, $N_{\xi^\frown 0}\subseteq N_\eta$ if and only if $N_{\xi^\frown 1}\subseteq N_\eta$. Also $N_{\xi^\frown 0}\subseteq \kappa^\kappa\backslash N_\eta$ if and only if $N_{\xi^\frown 1}\subseteq \kappa^\kappa\backslash N_\eta$. Thus $N_\eta$ is bad.
         
        {\bf Case $\omega<\alpha$.} Let $\xi:\omega\rightarrow\kappa$.
        \begin{itemize}
            \item {\bf Subcase $\xi^\frown 0,\xi^\frown 1\neq\eta\restriction\omega+1$.} Clearly $N_{\xi^\frown 0}\cup N_{\xi^\frown 1}\subseteq \kappa^\kappa\backslash N_\eta$.
            \item {\bf Subcase $\xi^\frown 0=\eta\restriction\omega+1$.} Clearly $N_\eta\subseteq N_{\xi^\frown 0}$ and $N_{\xi^\frown 1}\subseteq \kappa^\kappa\backslash N_\eta$. Thus $\circledast(N_\eta,\xi)$ does not hold.
            \item {\bf Subcase $\xi^\frown 1=\eta\restriction\omega+1$.} Similar as above, we conclude that $\circledast(N_\eta,\xi)$ does not hold.
        \end{itemize}
        Notice that only in the last two cases $N_{\xi^\frown 0}\cup N_{\xi^\frown 1}\not\subseteq N_\eta$ and $N_{\xi^\frown 0}\cup N_{\xi^\frown 1}\not\subseteq \kappa^\kappa\backslash N_\eta$. Therefore for almost all $\xi$, $N_{\xi^\frown 0}\cup N_{\xi^\frown 1}\subseteq N_\eta$ or $N_{\xi^\frown 0}\cup N_{\xi^\frown 1}\subseteq \kappa^\kappa\backslash N_\eta$, and $N_\eta$ is bad.
    
    It is clear that if $A$ is bad, then $\kappa\backslash A$ is bad. Let $\{A_i\}_{i<k}$ be bad sets, and $A=\cup_{i<\kappa}A_i$. Let $\eta\in \kappa^\omega$ such that there is $i<\kappa$ such that $\circledast(A_i,\eta)$ holds.
    \begin{itemize}
        \item If $N_{\eta^\frown 0}\cup N_{\eta^\frown 1}\subseteq A_i$, then 
        $\circledast(A,\eta)$.
        \item If for all $i<\kappa$, $N_{\eta^\frown 0}\cup N_{\eta^\frown 1}\not\subseteq A_i$, then $N_{\eta^\frown 0}\cup N_{\eta^\frown 1}\subseteq \cap_{i<\kappa}(\kappa\backslash A_i)=\kappa\backslash A$. Thus $\circledast(A,\eta)$.
    \end{itemize}
    
Finally, let $X=\{\eta\in\kappa^\kappa\mid 
\eta(\omega)>0\}.$ 
Clearly for all $\xi\in 2^\omega$, $N_{\xi^\frown 0}\cup N_{\xi^\frown 1}\not\subseteq X$ and $N_{\xi^\frown 0}\cup N_{\xi^\frown 1}\not\subseteq \kappa^\kappa\backslash X$. Since $2^\omega>\kappa$, $X$ is not bad, hence not $\kappa$-Borel\hspace{2pt}$^+$.
\end{example}

In order to avoid the pitfalls of Examples~\ref{example11} and \ref{example1} we redefine the topology of $\kappa^\kappa$  in a way which is reminiscent of the Tychonov- or product topology on $\kappa^\kappa$, when $\kappa$ is given the discrete topology. The new topology-like concept seems more natural than the bounded topology from the point of view of infinitary model theory, especially when  $\kappa^{<\kappa}>\kappa$. It coincides with the bounded topology in the case that 
$\kappa^{<\kappa}=\kappa$ (which was the context of e.g. \cite{MR3235820} and \cite{MR1242054}).

To make a judgement whether some concept is natural or not it is reasonable to use relevant applications as a criterion. Of course there is also an aesthetic aspect. As to applications, we look at infinitary model theory for inspiration and this leads us to make the below definition. However, the definition is aesthetically appealing as well, by virtue of its simplicity.

We now introduce the most important topology-like concept of this paper.

\begin{defn}\label{kappaopen}
    \emph{Basic $\kappa$-open} sets are the sets of the form $$N_\eta=\{\zeta\in \kappa^\kappa\mid\eta\subseteq \zeta\}$$ where $\eta:X\rightarrow \kappa$  and $X\subseteq \kappa$ has size less than $\kappa$; and the empty set. A set is $(\kappa,\lambda)$-open if it is a $\lambda$-union of (i.e. a union of $\lambda$)  basic $\kappa$-\emph{open} sets. A set is $(\kappa,\lambda)$-\emph{closed} if its complement is $(\kappa,\lambda)$-open. The class of $(\kappa,\lambda)$-\emph{Borel}  sets is the smallest class that contains all the basic $\kappa$-open sets and is closed under complements and $\lambda$-unions and $\lambda$-intersections.
Sets that are $(\kappa,\kappa)$-open (-closed, -Borel) are called simply $\kappa$-open (-closed, -Borel) sets.
\end{defn}

The family of all $\kappa$-open sets is not a topology if $\kappa<\kappa^{<\kappa}$ as it is then not closed under arbitrary unions. This seems like a setback. However, it is quite close to the product topology: To witt, let us denote by $\tau_p$ the \textit{product topology} (a.k.a. \textit{Tychonoff topology}) of $\kappa^\kappa$, i.e. basic open sets are of the form $N_\eta$, where $\eta:X\rightarrow\kappa$, $X\subseteq \kappa$ and $|X|<\omega$. 
%Similarly for $2^\kappa$. The topology $\tau_p$ on $2^\kappa$ has some nice properties, e.g. it is compact and regular. In particular, $2^\kappa$ with $\tau_p$ is separable if and only if $\kappa\leq 2^{\aleph_0}$. This and other properties of $\tau_p$ can be found in \cite[Proposition 3.12]{MR4409720}.
The close connection between $\kappa$-openness and $\tau_p$ is manifested by the equation $$N_\eta=\bigcap\{N_\zeta:\zeta\subseteq\eta, |\zeta|<\omega\}.$$ 
In particular, our family of $\kappa$-Borel sets is the same as the family of Borel subsets of $\tau_p$ (see Fact~\ref{2.7}). 

\begin{fact}Suppose $\kappa$ is regular.
    Let us denote by $B$ the class of Borel sets of the bounded topology, i.e. smallest class that contains all the open sets in the bounded topology and is closed under complements and $\kappa$-unions and $\kappa$-intersections. Every $\kappa$-Borel set is an element of $B$.
\end{fact}

\begin{cproof}
    Before we start, let us fix some notation. Given $\xi\in \kappa^{<\kappa}$, let us denote by $N_\xi$ the basic $\kappa$-open set defined by $\xi$ and $\hat N_\xi$ the basic open set defined by $\xi$ (i.e. the basic open set in the bounded topology). Notice that set theoretically, $\hat N_\xi$ and $N_\xi$ are the same object.

    By the definition of $\kappa$-Borel, it is enough to show that any basic $\kappa$-open set is an element of $B$. Let $X\subseteq \kappa$ be of size less than $\kappa$ and $\eta:X\rightarrow\kappa$. 
    For all $\alpha\in X$ let us define the function $\eta_\alpha:\{\alpha\}\rightarrow\{\eta(\alpha)\}$. Let us denote by $N_\alpha$ the basic open set defined by $\eta_\alpha$ and define $$A_\alpha=\{\zeta\in \kappa^{<\kappa}\mid \eta_\alpha\subseteq \zeta, \dom (\zeta)=\alpha+1\}.$$
    Clearly $$N_\alpha=\bigcup_{\zeta\in A_\alpha}\hat N_\zeta.$$    
    Notice that the right side is a union of open sets in the bounded topology, thus $N_\alpha$ is an open set (in the bounded topology), independent of the size of $A_\alpha$, and therefore an element of $B$.
    Finally $N_\eta=\bigcap_{\alpha\in X}N_\alpha$ and since $X$ has size less than $\kappa$, $N_\eta\in B$.
\end{cproof}

Note that the total number of basic $\kappa$-open sets is $\kappa^{<\kappa}$, so  $(\kappa,\kappa^{<\kappa})$-open  is equivalent to open  in the usual sense of  Generalized Baire Spaces, as is $\kappa$-open if 
$\kappa^{<\kappa}=\kappa$. On the other hand, if $\kappa^{<\kappa}>\kappa$, then we get a difference:

\begin{prop}
    Suppose $2^\omega>\kappa$. Then there is an open $\kappa$-Borel set which is  not $\kappa$-open.  
\end{prop}

\begin{cproof}
    Let $X$ be the set of all strictly increasing functions $f:\omega\rightarrow\kappa$ and $U=\bigcup_{f\in X}N_f$. It is clear that $U$ is open. Since obviously $U=\bigcap_{n<\omega}(\bigcup_{f\in X}N_{f\restriction n})$, $U$ is $\kappa$-Borel. Let us show that $U$ is not $\kappa$-open. Assume, towards contradiction, that there is $Z\subseteq \kappa^\kappa$ such that $|Z|\leq \kappa$ and $U=\bigcup_{\eta\in Z}N_\eta$. Notice that for all $\eta\in Z$, $\omega\subseteq dom(\eta)$, otherwise there is $\xi\in N_\eta$ such that $\xi\restriction \omega$ is not strictly increasing. It is easy to see that $|X|=\kappa^\omega>\kappa$, so there is $f\in X$ such that $f\not \subseteq \eta$, for all $\eta\in Z$. Thus $N_f\not\subseteq \bigcup_{\eta\in Z}N_\eta=U$, a contradiction.
\end{cproof}

\begin{comment}

We say that a set is $(\kappa,\lambda)$-comeager if it contains an intersection of $\lambda$ many $(\kappa,\lambda)$-open dense sets. 

\begin{thm}
    Suppose $2^\omega=2^{\omega_1}=\omega_2$. If $X$ is $(\omega,\omega_1)$-Borel, then there are an $(\omega,\omega_1)$-open set, $U$, and an $(\omega,\omega_1)$-comeager set $G$ such that $G$ is $(\omega,\omega_1)$-Borel and $X\cap G= U\cap G$. Even more, MA implies that $(\omega,\omega_1)$-comeager sets are dense.
\end{thm}
\begin{proof}
    apu 182
\end{proof}

\end{comment}

Notice that the space $2^\kappa$ itself is a $\kappa$-Borel subset of $\kappa^\kappa$. This observation will be useful later.

Let us denote by Borel\hspace{1pt}$_p$ the smallest class of sets that contains all the open sets of the product topology $\tau_p$ and is closed under complement, $\kappa$-unions, and $\kappa$-intersections. This is the analogue of classical Borel sets in the framework of generalized Baires spaces with product topology. We refer to the elements of Borel\hspace{1pt}$_p$ simply as Borel\hspace{1pt}$_p$ sets. 

\begin{fact}\label{2.7}
    A set $A\subseteq \kappa^\kappa$ is a Borel\hspace{1pt}$_p$ set if and only if it is  $\kappa$-Borel.
\end{fact}
\begin{cproof}
    $\Rightarrow$). From the definition of $\kappa$-Borel, it is clear that the basic open sets of $\tau_p$ are basic $\kappa$-open sets of $\tau_p$. Let $A$ be an open set of $\tau_p$, so $\bigcup_{i\in I}A_i$ where $A_i$ are basic open sets. Since there are at most $\kappa$ basic open sets in $\tau_p$, $A$ is a $\kappa$-open set. The proof follows from the definitions of Borel\hspace{1pt}$_p$ and $\kappa$-Borel.

    $\Leftarrow$). Let $N_\eta$ be a basic $\kappa$-open set, $\eta:X\rightarrow \kappa$ where $|X|<\kappa$. For all $\alpha\in X$, let us define $N_\alpha$ to be the basic $\kappa$-open set given by the function $f:\{\alpha\}\rightarrow\{\eta(\alpha)\}$. Thus $N_\eta=\bigcap_{\alpha\in X}N_\alpha$ and it is Borel\hspace{1pt}$_p$. The proof follows from the definitions of Borel\hspace{1pt}$_p$ and $\kappa$-Borel.
\end{cproof}

We define basic $\kappa$-open sets in the cartesian product $\kappa^\kappa\times\kappa^\kappa$ in a similar way. Thus, the basic $\kappa$-open sets of the product are of the form $N_\eta\times N_\xi$, where $N_\eta$ and $N_\xi$ are basic $\kappa$-open sets of $\kappa^\kappa$. A set is $(\kappa,\lambda)$-open in $\kappa^\kappa\times\kappa^\kappa$ if it is a $\lambda$-union of basic $\kappa$-open sets. A set is $(\kappa,\lambda)$-closed in $\kappa^\kappa\times\kappa^\kappa$ if its complement is $(\kappa,\lambda)$-open. The class of $(\kappa,\lambda)$-Borel subsets of $\kappa^\kappa\times\kappa^\kappa$ is the smallest class that contains all the basic open sets and is closed under complements, $\lambda$-unions, and $\lambda$-intersections. Similarly for the space $2^\kappa$.

\section{$\kappa$-\borelstar\ and $\kappa$-$\Sigma^1_1$}

The concept of a \borelstar-set was introduced in \cite{MR1242054} as a generalization of the concept of a Borel-set. We will redefine \borelstar\ sets here anew as well, as the definition in \cite{MR1242054} calls for a slight modification when we move away from the  bounded topology. There is no change if $\kappa^{<\kappa}=\kappa$. In classical Polish spaces every Borel set has a Borel-code which is a countable tree without infinite branches. \borelstar-sets  arise when the Borel-code is allowed to have infinite branches. We first define the appropriate trees and then define what it means for such a tree to be a \borelstar-code.

\begin{defn}\label{labelled_tree_star}
A pair $(T,L)$ is a \emph{good labelled $(\kappa,\lambda)$-tree} if 
\begin{itemize}
\item [(i)] $T=(T,\le)$ is a tree without $\kappa$-branches.
\item [(ii)] Every element of $T$ has $\le\kappa$ many immediate successors.
\item [(iii)] $|T|\le\lambda$.
\item [(iv)] Every increasing sequence in $T$ has a supremum.
\item [(v)] If $t\in T$ is not a leaf (i.e. maximal element), then $L(t)\in \{\bigcup,\bigcap\}$.
\item [(vi)]  If $t\in T$ is  a leaf, then $L(t)$ is a basic $\kappa$-open set.
\end{itemize}
\end{defn}

Sometimes we use $(\kappa,\lambda)$-Borel sets as labels. It is clear how this fits in the general definition.

\begin{defn}\label{Borel_star}
A good labelled $(\kappa,\lambda)$-tree is a \emph{$(\kappa,\lambda)$-\borelstar-code} if 
\begin{itemize}
\item [(i)] There is a function $\pi$ such that $\dom(\pi)$ is a $\kappa$-closed subset of $\kappa^\kappa$ and for all $\eta\in \dom(\pi)$, $\pi(\eta)$ is a strategy of II in the usual (see \cite{MR1242054}, \cite{MR3753133}) \borelstar-game $GB(\xi,(T,L))$, where $\xi\in\kappa^\kappa$.
\item [(ii)] For all $\xi\in\kappa^\kappa$, if II has a winning strategy in $GB(\xi,(T,L))$, denoted $II\uparrow GB(\xi,(T,L))$, then there is some $\eta\in\dom(\pi)$ such that $\pi(\eta)$ is a winning strategy of II in $GB(\xi,(T,L))$.
\item [(iii)] The set $$\begin{array}{l}
\{(\xi,\eta) : \eta\in\dom(\pi)\text{ and }\pi(\eta)\mbox{ is a winning strategy}\\ \qquad\qquad\text{of II in }GB(\xi,(T,L))\}\end{array}$$ is $\kappa$-Borel.
\end{itemize} A $(\kappa,\kappa)$-\borelstar-code is called a \emph{$\kappa$-\borelstar-code}.
\end{defn}

Now that we have defined $(\kappa,\lambda)$-\borelstar-codes, we can define $(\kappa,\lambda)$-\borelstar-sets as such sets that have a $(\kappa,\lambda)$-\borelstar-code:

\begin{defn}\label{def_Borel_star}
A set $X\subseteq \kappa^\kappa$ is \emph{$(\kappa,\lambda)$-\borelstar} if there is a $(\kappa,\lambda)$-\borelstar-code $(T,L)$  such that for all $\xi\in\kappa^\kappa$: 
$$\xi\in X\leftrightarrow II\uparrow GB(\xi,(T,L)).$$ A set is \emph{$\kappa$-\borelstar} if it is $(\kappa,\kappa)$-\borelstar.
\end{defn}

\begin{lemma}\label{labelled_to_Borel}
Suppose $\lambda\leq\kappa$. 
If $(T,L)$ is a good labelled $(\kappa,\lambda)$-tree, then it is a $(\kappa,\lambda)$-{\borelstar}-code.
\end{lemma}

\begin{cproof}

Let $X=\{t\in T\mid L(t)=\cup\}$ and $F:X\rightarrow\kappa$ be an injective function. For all $t\in X$, let $G_t$ be an injective function from the immediate successors of $t$ to $\kappa$. Let $\dom(\pi)$ be the set of all $\eta\in \kappa^\kappa$ such that for all $t\in X$, $\eta(F(t))\in \rng(G_t)$. It is clear that $\dom(\pi)$ is $\kappa$-closed. Let $\pi(\eta)$ be the strategy that at $t\in X$, it chooses $G^{-1}_t(\eta(F(t)))$. Clearly (i) and (ii) are satisfied.

Let $Y=\{t\in T\mid t \text{ is a leaf}\}$, since $\lambda\leq\kappa$ implies $|T|\leq\kappa$, $Y$ has power at most $\kappa$. For all $t\in Y$, let $\eta_t$ be such that $L(t)=N_{\eta_t}$. For all $t\in Y$, there are $Y_t\subseteq \kappa$, $|Y_t|<\kappa$, and $\xi_t:Y_t\rightarrow\kappa$ such that for all $\vartheta\in \dom(\pi)$, $\xi_t\subseteq \vartheta$ if and only if for all $\zeta\in \kappa^\kappa$, there is a play in $GB(\zeta,(T,L))$ in which II has used $\pi(\vartheta)$ and it ends at $t$.

We conclude that for all $(\zeta,\vartheta)\in (\kappa^\kappa\backslash N_{\eta_t})\times (\dom(\pi)\cap N_{\xi_t})$, $\pi(\vartheta)$ is not a winning strategy of II in $GB(\zeta,(T,L))$. Thus $$H=\bigcup_{t\in Y}(\kappa^\kappa\backslash N_{\eta_t})\times (\dom(\pi)\cap N_{\xi_t})$$ is $\kappa$-Borel. Notice that $$\{(\xi,\eta) : \eta\in\dom(\pi)\text{ and }\pi(\eta)\mbox{ is a winning strategy of II in }GB(\xi,(T,L))\}$$ is equal to $(\kappa^\kappa\times\kappa^\kappa\backslash\ H)\cap \left(\kappa^\kappa\times \dom(\pi)\right)$, the intersection of two $\kappa$-Borel sets.
Therefore $$\begin{array}{l}
\{(\xi,\eta) : \eta\in\dom(\pi)\text{ and }\pi(\eta)\mbox{ is a winning strategy}\\ \qquad\qquad\text{of II in }GB(\xi,(T,L))\}\end{array}$$ is $\kappa$-Borel and (iii) holds.

\begin{comment}
Since $|T|\leq \lambda$, it is possible to choose $\pi$ continuous so that every strategy has a code given by $\pi$. Thus (i) and (ii) are satisfied.  
To show (iii), let $t\in T$ be a leaf and $b_t=\{x\in T\mid x<t\}$ the branch of $T$ with leaf $t$. There is $X_t\subseteq\kappa$, $|X_t|<\kappa$, and $\zeta_t\in \kappa^{X_t}$, such that $L(t)=N_{\zeta_t}$. Since $(T,L)$ is a good labelled $(\kappa,\lambda)$-tree, by Definition \ref{labelled_tree_star} (iii), $|b_t|\leq \lambda$. 
So there is $Y_t\subseteq \kappa$, $|Y_t|<\kappa$, and $\eta_t\in \kappa^{Y_t}$ such that for all $\vartheta\in N_{\eta_t}$ and $\xi\in \kappa^\kappa\backslash N_{\zeta_t}$, there is a strategy $\sigma$ such that if I plays according to $\sigma$ and II plays according to $\pi(\vartheta)$ in the game $GB(\xi,(T,L))$, the game ends at $t$.
\end{comment}
\end{cproof}

It follows that if $\kappa^{<\kappa}=\kappa$ and $\lambda\ge\kappa$, then $(\kappa,\lambda)$-\borelstar\ implies $\kappa$-\borelstar, and $\kappa$-\borelstar\ is the same as \borelstar\ as defined in \cite{MR3235820} and \cite{MR1242054}.  

\begin{cor}\label{Borel*_proj_closed}
    Suppose $\lambda\leq\kappa$ and $X\subseteq \kappa^\kappa$ is a \emph{$(\kappa,\lambda)$-\borelstar} set. Then $X$ is the projection of a $\kappa$-closed set. 
\end{cor}
\begin{cproof}
By the definition of \emph{$(\kappa,\lambda)$-\borelstar} set, there is a $(\kappa,\lambda)$-\borelstar-code $(T,L)$ such that for all $\xi\in\kappa^\kappa$: 
$$\xi\in X\leftrightarrow II\uparrow GB(\xi,(T,L)).$$
By definition, $(T,L)$ is a good labelled $(\kappa,\lambda)$-tree.
Let $\pi$ be the function constructed in Lemma \ref{labelled_to_Borel} to show that $(T,L)$ is a $(\kappa,\lambda)$-\borelstar-code. Let us show that \begin{equation}\label{is a winning}
    \begin{array}{l}
 \{(\xi,\eta) : \eta\in\dom(\pi)\text{ and }\pi(\eta)\mbox{ is a winning strategy}\\ \qquad\qquad\text{of II in }GB(\xi,(T,L))\}\end{array}\end{equation} is $\kappa$-closed. Let $Y$, $\eta_t$, and $\xi_t$ as in Lemma \ref{labelled_to_Borel}.

    Notice that the complement of basic $\kappa$-open sets are $\kappa$-unions of basic $\kappa$-open sets, i.e. are $\kappa$-open. Thus the sets $$(\kappa^\kappa\backslash N_{\zeta_t})\times N_{\eta_t}$$ from the proof of Lemma \ref{labelled_to_Borel}, are $\kappa$-unions of $\kappa$-open sets. Thus $$A=\bigcup_{t\in Y}(\kappa^\kappa\backslash N_{\zeta_t})\times N_{\eta_t}$$ is $\kappa$-open.

    Notice that (\ref{is a winning}) 
    %$$\{(\xi,\eta) : \eta\in\dom(\pi)\text{ and }\pi(\eta)\mbox{ is a winning strategy of II in }GB(\xi,(T,L))\}$$ 
    is equal to $(\kappa^\kappa\times\kappa^\kappa\backslash\ A)\cap \left(\kappa^\kappa\times \dom(\pi)\right)$, the intersection of two $\kappa$-closed sets. Thus the set (\ref{is a winning})
    %$$\{(\xi,\eta) : \eta\in\dom(\pi)\text{ and }\pi(\eta)\mbox{ is a winning strategy of II in }GB(\xi,(T,L))\}$$ 
    is $\kappa$-closed.
\end{cproof}

We now define the important concept of $\kappa$-$\Sigma^1_1$. In view of Lemma~\ref{dead-end} one has to define this concept with care.

\begin{defn}
We say that a set is $\kappa$-$\Sigma^1_1$ if it is the projection of a $\kappa$-Borel set. A set is $\kappa$-$\Pi^1_1$ if it is the complement of a $\kappa$-$\Sigma^1_1$. A set is $\kappa$-$\Delta^1_1$ if it is both $\kappa$-$\Sigma^1_1$  and $\kappa$-$\Pi^1_1$ 
\end{defn}

Notice that there are only $2^\kappa$ many $\kappa$-$\Sigma^1_1$ sets, but there are $2^{2^\kappa}$ subsets of $\kappa^\kappa$. Hence most subsets of $\kappa^\kappa$ are not $\kappa$-$\Sigma^1_1$.  It is clear from Definition \ref{Borel_star} (iii), that if $X\subseteq \kappa^\kappa$ is $\kappa$-\borelstar, then it is $\kappa$-$\Sigma_1^1$.

\begin{defn}\label{Borel_hier2}
Let us define the following hierarchy.
\begin{itemize}
\item $\kappa$-$\Sigma_0$ is the set of basic $\kappa$-open sets of the form $N_{\eta}$, $|\eta|=1$.
\item If $\alpha$ is even, then $\kappa$-$\Sigma_{\alpha+1}$ is the set of all $\kappa$-unions of $\kappa$-$\Sigma_{\alpha}$-sets.
\item If $\alpha$ is odd, then $\kappa$-$\Sigma_{\alpha+1}$ is the set of all $\kappa$-intersections of $\kappa$-$\Sigma_{\alpha}$-sets.
\item If $\alpha$ is limit, then $\kappa$-$\Sigma_{\alpha}=\cup_{\beta<\alpha}\kappa\text{-}\Sigma_{\beta}$.
\end{itemize}
\end{defn}

\begin{comment}
Notice that in this hierarchy, the $\alpha$-th level is different from the one in Definition \ref{Borel_hier1}.
It is easy to see that for all $\alpha<\kappa^+$ there are $\alpha<\beta_1<\beta_2$ and $\alpha<\gamma_1<\gamma_2$ such that $\kappa\text{-}\Sigma_{\alpha}\subseteq\kappa\text{-}\Sigma^0_{\beta_1}\subseteq\kappa\text{-}\Sigma_{\beta_2}$ and $\kappa\text{-}\Sigma_{\alpha}\subseteq\kappa\text{-}\Pi^0_{\gamma_1}\subseteq\kappa\text{-}\Sigma_{\gamma_2}$

\begin{defn}\label{Borel_hier1}
Let us define the following hierarchy.
\begin{itemize}
\item $\kappa$-$\Sigma_0^0=\kappa$-$\Pi_0^0$ the set of basic $\kappa$-open sets.
\item $\kappa$-$\Sigma^0_\alpha=\{\bigcup_{\gamma<\kappa}A_\gamma\mid A_\gamma\in \bigcup_{0\leq\beta<\alpha}\kappa\text{-}\Pi^0_\beta\}$.
\item $\kappa$-$\Pi^0_\alpha=\{\kappa^\kappa\backslash X\mid X\in \kappa\text{-}\Sigma^0_\alpha\}$.
\end{itemize}
\end{defn}
\end{comment}

Clearly $\kappa$-Borel$=\bigcup_{\alpha<\kappa^+}\kappa\text{-}\Sigma_\alpha$. 

\begin{defn}
We call a function $f\colon\kappa^\kappa\rightarrow\kappa^\kappa$ \emph{$\kappa$-continuous} if the inverse image of a $\kappa$-open set is a $\kappa$-open set.    
\end{defn}

\begin{fact}\label{funcion_borel}
    If $f\colon\kappa^\kappa\rightarrow\kappa^\kappa$ is a $\kappa$-continuous function, then for all $\kappa$-Borel $X\subseteq \kappa^\kappa$, $f^{-1}[X]$ is $\kappa$-Borel.
\end{fact}
\begin{cproof}
    Let us proceed by induction over $\kappa\text{-}\Sigma_\alpha$, from Definition \ref{Borel_hier2}. Since $f$ is $\kappa$-continuous, if $X\in \kappa\text{-}\Sigma_0$, then $f^{-1}[X]$ is $\kappa$-open. Thus $X$ is $\kappa$-Borel. Let us suppose that $\alpha<\kappa^+$ is such that for all $\beta<\alpha$, if $X\in \kappa\text{-}\Sigma_\beta$, then $f^{-1}[X]$ is $\kappa$-Borel. It is clear that if $\alpha$ is limit, then for all $X\in \kappa\text{-}\Sigma_\alpha=\cup_{\beta<\alpha}\kappa\text{-}\Sigma_{\beta}$, $f^{-1}[X]$ is $\kappa$-Borel.
    
    Let us suppose $\alpha=\beta+1$ for some $\beta<\kappa^+$ even (the odd case is similar). Let $X\in \kappa\text{-}\Sigma_\alpha$. So, $X=\bigcup_{\gamma<\kappa}A_\gamma$, where $A_\gamma\in \kappa\text{-}\Sigma_\beta$. Since $A_\gamma$ is $\kappa$-Borel and $f^{-1}[X]=\bigcup_{\gamma<\kappa}f^{-1}[A_\gamma]$, we conclude that $X$ is $\kappa$-Borel.
\end{cproof}

%Notice that $\kappa$-$\Pi^0_\alpha=\{\bigcap_{\gamma<\kappa}A_\gamma\mid A_\gamma\in \bigcup_{0\leq\beta<\alpha}\kappa\text{-}\Sigma^0_\beta\}$ and for all $\alpha$ limit $\bigcup_{\beta<\alpha}\kappa\text{-}\Sigma^0_\beta=\bigcup_{\beta<\alpha}\kappa\text{-}\Pi^0_\beta$.
%We define the $\kappa$-Borel rank of a $\kappa$-Borel set,$A$, as the least ordinal $\alpha$, such that $A\in \kappa\text{-}\Sigma^0_\alpha\cup \kappa\text{-}\Pi^0_\alpha $. We denote the $\kappa$-Borel rank of $A$ by $rk_B(A)$.

For all $t\in T$ we will denote by $t^-$ the immediate
predecessor of $t$ (in case it exists), i.e. the unique element of $T$ that satisfies:
 \begin{itemize}
     \item $t^-<t$,
     \item there is no $t''\in T$ such that $t^-<t''<t$.
 \end{itemize}
 
 \begin{defn}\label{rank_tree}
     Let $T$ be a tree without infinite branches. For all $t\in T$, we define $rk(t)$ as follows:
     \begin{itemize}
         \item If $t$ is a leaf, then $rk(t)=0$;
         \item if $t$ is not a leaf, then $rk(t)=\cup\{rk(t')+1\mid t'^-=t\}$.
         \item If $T$ is not empty and has a root, $r$, then the rank of $T$ is denoted by $rk(T)$ and is equal to $rk(r)$.
     \end{itemize}
 \end{defn}

    \begin{lemma}\label{Borel_is_Borel*}
        %Suppose $\kappa$ is such that $\kappa^{\omega}=\kappa$. 
        A set $X\subseteq \kappa^\kappa$ is  $\kappa$-Borel if and only if there is a  $\kappa$-\borelstar-code $(T,L)$  such that $T$ is a subtree of $\kappa^{<\omega}$ and for all $\xi\in\kappa^\kappa$: 
$$\xi\in X\leftrightarrow II\uparrow GB(\xi,(T,L)).$$
    In particular, then every $\kappa$-Borel set is a $\kappa$-\borelstar-set.
    \end{lemma}
    
\begin{cproof}
$\Rightarrow)$ We will show by induction over $\alpha$ that for every $X\in \kappa\text{-}\Sigma_\alpha$, the statement holds.

The result is trivial when $\alpha=0$. 

Suppose $\alpha$ is such that for all $\beta\leq\alpha$ and $X\in \kappa\text{-}\Sigma_\beta$, there is a $\kappa$-\borelstar-code $(T,L)$ such that $(T,L)$ codes $X$ and $T$ is a subtree of $\kappa^{<\omega}$. 
Let us show the case where $\alpha$ is even. The odd case is similar.
Suppose $X\in \kappa$-$\Sigma_{\alpha+1}$, so $X=\bigcup_{\gamma<\kappa}A_\gamma$, where $A_\gamma\in \kappa$-$\Sigma_\alpha$. By the induction hypothesis, there are $\kappa$-\borelstar-codes $\{(T_\gamma,L_\gamma)\}_{\gamma<\kappa}$ such that $(T_\gamma,L_\gamma)$ codes $A_\gamma$ and $T_\gamma$ is a subtree of $\kappa^{<\omega}$, for all $\gamma$. 
Let $\mathcal{T}=\{r\}\cup\bigcup_{\gamma<\kappa}T_\gamma\times\{\gamma\}$ be the tree ordered by:
\begin{itemize}
    \item $r<_\mathcal{T}(x,j)$ for all $(x,j)\in \bigcup_{\gamma<\kappa}T_\gamma\times\{\gamma\}$,
    \item $(x,\gamma)<_\mathcal{T}(y,j)$ if and only if $\gamma=j$ and $x<_{T_\gamma}y$ in $T_\gamma$.
\end{itemize}

Let $T\subseteq \kappa^{<\omega}$ be a tree isomorphic to $\mathcal{T}$ and let $\mathcal{G}:T\rightarrow \mathcal{T}$ be a tree isomorphism. 
If $\mathcal{G}(x)\neq r$, then denote $\mathcal{G}(x)$ by $(\mathcal{G}_1(x),\mathcal{G}_2(x))$. Define $L$ by $L(x)=\cup$ if $G(x)=r$, and $L(x)=L_{\mathcal{G}_2(x)}(\mathcal{G}_1(x))$. 

Let us show that $(T,L)$ codes $X$. Let $\eta\in X$, so there is $\gamma<\kappa$, such that $\eta\in A_\gamma$. 
II starts by choosing $\mathcal{G}^{-1}(x,\gamma)$, where $x$ is the root of $T_\gamma$. II continues playing with the winning strategy from the game $GB(\eta,(T_\gamma,L_\gamma))$, choosing the element given by $\mathcal{G}^{-1}$. We conclude that $II\uparrow GB(\eta,(T,L))$. 

Let $\eta\not\in X$, so for all $\gamma<\kappa$, $\eta\not\in A_\gamma$, so II has no winning strategy for the game $GB(\eta,(T_\gamma,L_\gamma))$. Thus II cannot have a winning strategy for the game $GB(\eta,(T,L))$.

Clearly $(T,L)$ is a good labelled $(\kappa,\kappa)$-tree. By Lemma \ref{labelled_to_Borel}, $(T,L)$ is a $\kappa$-\borelstar-code.

\begin{comment}
Let us code the winning strategies by a function $\pi$, such that the set $$W=\{(\xi,\eta) : \pi(\eta)\mbox{ is a winning strategy of II in }GB(\xi,(T,L))\}$$ is $\kappa$-Borel.

Since for all $\gamma<\kappa$, $(T_\gamma,L_\gamma)$ is a $\kappa$-\borelstar-code, there is a function $\pi_\gamma$ such that $$W_\gamma=\{(\xi,\eta) : \pi_\gamma(\eta)\mbox{ is a winning strategy of II in }GB(\xi,(T_\gamma,L_\gamma))\}$$ is $\kappa$-Borel. 
Let us define $p:\kappa^\kappa\rightarrow\kappa^\kappa$ by $$p(\eta)(\alpha)=\begin{cases}     \eta(\alpha) & \mbox{if } \alpha\ge\omega,\\
        \eta(\alpha+1) &\mbox{if } \alpha<\omega.\end{cases}$$  
        
Let $\pi(\eta)$ be the strategy in which II starts by choosing $\eta(0)$ and then plays following the strategy $\pi_{\eta(0)}(p(\eta))$. 
Let $g$ be the map $(\xi,\eta)\mapsto (\xi,p(\eta))$. Clearly $p$ is $\kappa$-continuous, thus $g$ is $\kappa$-continuous. Therefore $g^{-1}[W_\gamma]$ is $\kappa$-Borel. 

Notice that $\pi(\eta)$ is a winning strategy of II in $GB(\xi,(T,L))$ if and only if $\pi_{\eta(0)}(p(\eta))$ is a winning strategy of II in $GB(\xi,(T_{\eta(0)},L_{\eta(0)}))$. Since $W=\bigcup_{\gamma<\kappa}g^{-1}[W_\gamma]$ we conclude that $W$ is $\kappa$-Borel.
\end{comment}

The case $\alpha$ limit is similar to the previous one, since $\kappa$-$\Sigma_{\alpha}=\cup_{\beta<\alpha}\kappa\text{-}\Sigma_{\beta}$ when $\alpha$ is limit.

$\Leftarrow)$ We will use induction over the rank of $T$, to show that if $(T,L)$ is a good labelled $(\kappa,\kappa)$-tree such that $T$ is a subtree of $\kappa^{<\omega}$, then the set $X$ such that for all $\xi\in\kappa^\kappa$: 
$$\xi\in X\leftrightarrow II\uparrow GB(\xi,(T,L)),$$ is a $\kappa$-Borel set. Notice that this implies that if $(T,L)$ is a $\kappa$-\borelstar-code where $T$ is a subtree of $\kappa^{<\omega}$, then it codes a $\kappa$-Borel set.

If $rk(T)=0$, then $T$ has only one node $r$, thus $X=L(r)$ and $X$ is a basic $\kappa$-open set. Let $\alpha<\kappa^+$ be such that for all good labelled $(\kappa,\kappa)$-tree $(T',L')$ with $T'$ a subtree of $\kappa^{<\omega}$ and $rk(T')<\alpha$, $(T',L')$ codes a $\kappa$-Borel set. 
 
Let $(T,L)$ be a  good labelled $(\kappa,\kappa)$-tree  such that $T$ is a subtree of $\kappa^{<\omega}$, with $rk(T)=\alpha$, and $X$ is such that for all $\xi\in\kappa^\kappa$: 
$$\xi\in X\leftrightarrow II\uparrow GB(\xi,(T,L)).$$ Let $B=\{t\in T\mid t^-=r\}$, where $r$ is the root of $T$. Notice that $|B|\leq \kappa$ by Definition \ref{labelled_tree_star} (ii).
For all $t\in B$, define the good labelled $(\kappa,\kappa)$-tree $(T_t,L_t)$ as follows:
\begin{itemize}
    \item $T_t=\{x\in T\mid t\leq x\}$,
    \item $L_t=L\restriction T_t$.
\end{itemize}
Since $rk(T)=\alpha$, for all $t\in B$, $rk(T_t)<\alpha$. By the induction hypothesis, for all $t\in B$ the set $X_t$ of all $\xi\in\kappa^\kappa$ such that 
$II\uparrow GB(\xi,(T_t,L_t))$, is a $\kappa$-Borel set.

It is easy to see that, if $L(r)=\cup$, then $X=\cup_{t\in B}X_t$. Also, if $L(r)=\cap$, then $X=\cap_{t\in B}X_t$.

Since the class of $\kappa$-Borel sets is closed under unions and intersections of length $\kappa$, the proof follows. 
\end{cproof}
\begin{lemma}
    Every $\kappa$-Borel set is the projection of a $\kappa$-closed set (and hence $\kappa$-$\Sigma^1_1$).
\end{lemma}
\begin{cproof}
It follows from Corollary \ref{Borel*_proj_closed} and Lemma \ref{Borel_is_Borel*}.
\end{cproof}

\begin{cor} The
    $\kappa$-$\Sigma^1_1$ sets are projections of $\kappa$-closed sets.
\end{cor}

\begin{cproof}
    Let $A$ be a $\kappa$-$\Sigma_1^1$ set. Thus, $A$ is the projection of a $\kappa$-Borel set $B$. By the previous Lemma, $B$ is the projection of a $\kappa$-closed set. Thus $A$ is the projection of a $\kappa$-closed set.
\end{cproof}

We now define another variety of Borelness which turns out to be useful for us. 

\begin{defn}
\mbox{}
\begin{itemize}
    %\item The class of \emph{weak $\kappa$-Borel} sets is the smallest class that contains all the open sets and is closed under complements, $\kappa$-unions, and $\kappa$-intersections.
    \item     We define \emph{weak good labelled $\kappa$-trees} as in Definition \ref{labelled_tree_star} by omitting (iii).
    \item     We define  \emph{weak $\kappa$-\borelstar-codes} as in Definition \ref{Borel_star} by using weak good labelled $\kappa$-trees so that in (i) in addition we require that $dom(\pi)=\kappa^\kappa$ and in (iii) we require that the displayed set is closed.
    \item     We define \emph{weak $\kappa$-\borelstar} sets as in Definition \ref{def_Borel_star} using weak $\kappa$-\borelstar-codes.
\end{itemize}

Note that the weak good labelled $\kappa$-trees are exactly the good labelled $(\kappa,2^\kappa)$-trees.

%    A set $X$ is weak $\kappa$-\borelstar if it is weak $(\kappa,\kappa)$-\borelstar.
\end{defn}

\begin{lemma}
    Suppose $2^\omega=2^{\kappa}$, where $\kappa>\omega$. Then every $A\subseteq \kappa^\kappa$ is a weak $\kappa$-\borelstar set.
\end{lemma}
\begin{cproof}
    Let $\Pi:\kappa^\omega\rightarrow\kappa^\kappa$ be a bijection and $A\subseteq\kappa^\kappa$. Let $T=\kappa^{\leq\omega+1}$, and define $L$ as follows:
    \begin{itemize}
        \item $L(\eta)=\cup$ if $dom(\eta)<\omega$;
        \item $L(\eta)=\cap$ if $dom(\eta)=\omega$;
        \item If $dom(\eta)=\omega+1$, then 
        $$L(\eta)=\begin{cases}     N_{\Pi(\eta\restriction\omega)\restriction\eta(\omega)} &\mbox{if } \Pi(\eta\restriction\omega)\in A,\\
        \emptyset &\mbox{otherwise. }\end{cases}$$  
    \end{itemize}
    Clearly $(T,L)$ is a weak good labelled $\kappa$-tree. Let us show that $$\xi\in A\leftrightarrow II\uparrow GB(\xi,(T,L)).$$ Let $\xi$ be such that $II\uparrow GB(\xi,(T,L))$, then there is $\eta\in T$ with $dom(\eta)=\omega+1$, such that $\xi\in L(\eta)$. Thus $L(\eta)\neq\emptyset$ and $\Pi(\eta\restriction\omega)\in A$. Clearly $\xi\in \bigcap_{\alpha<\kappa}N_{\Pi(\eta\restriction\omega)\restriction\alpha}$ if and only if $\xi=\Pi(\eta\restriction\omega)$. Since $\Pi(\eta\restriction\omega)\in A$, $\xi\in A$. 

    Let $\xi\in A$. There is $\eta\in \kappa^\omega$ such that $\Pi(\eta)=\xi$. $II$ could play such that after $\omega$ moves the game is at $\eta$. Since $\Pi(\eta)=\xi$, it doesn't matter where $I$ moves, $II$ will win.

     Let $\pi(\zeta)$ be the strategy in which II copies $\zeta\restriction\omega$. Let us show that $$W=\{(\xi,\zeta) : \pi(\zeta)\mbox{ is a winning strategy of II in }GB(\xi,(T,L))\}$$ is 
     %a weak $\kappa$-Borel set. 
     closed. Notice that $W=\{(\xi,\zeta)\mid \xi=\Pi(\zeta\restriction\omega)\}$. Let $\xi,\zeta\in \kappa^\kappa$ be such that $(\xi,\zeta)\notin W$. Thus $\xi\neq\Pi(\zeta\restriction\omega)$. Let $\alpha<\kappa$ be such that $\xi(\alpha)\neq\Pi(\zeta\restriction\omega)(\alpha)$. Thus for all $(\xi',\zeta')\in N_{\xi\restriction\{\alpha\}}\times N_{\zeta\restriction\omega}$, $(\xi',\zeta')\notin W$. We conclude that $W$ is indeed closed.
    
\end{cproof}
Notice that $(T,L)$ in the previous lemma is not a $\kappa$-\borelstar-code.

We now prove a kind of Souslin-Kleene Theorem for $\kappa^\kappa$.
\begin{lemma}
%If $X\subseteq\kappa^\kappa$ is $\kappa$-\borelstar, then it is $\kappa$-$\Sigma^1_1$; and conversely, i
If 
 $X\subseteq\kappa^\kappa$ is $\kappa$-$\Delta^1_1$, then it is $(\kappa,\kappa^{<\kappa})$-\borelstar.
\end{lemma}

\begin{cproof}

Let $X\subseteq \kappa^\kappa$ be a $\kappa$-$\Delta_1^1$. Let $Y$ and $Z$ be $\kappa$-closed such that $X=pr_1(Y)$ and $\kappa\backslash X=pr_1(Z)$.

Let $T$ consist of triples $(x_0,x_1,x_2)$ that satisfies one of the following:
\begin{itemize}
    \item $(\eta_0,\eta_1,\eta_2)\in (\kappa^{<\kappa})^3$ such that one of the following holds:
    \begin{itemize}
        \item $\dom(\eta_0)=\dom(\eta_1)=\dom(\eta_2)$;
        \item $\dom(\eta_0)=\dom(\eta_1)=\dom(\eta_2)+1$.
    \end{itemize}
    The triples are ordered by coordinate wise inclusion. 
    A triplet $(\eta'_0,\eta'_1,\eta'_2)$ is not a leaf %$(\eta'_0,\eta'_1,\eta'_2)< (\eta_0,\eta_1,\eta_2)$, 
    if and only if $N_{(\eta_0',\eta_1')}\cap Z\neq\emptyset$ and $N_{(\eta_0',\eta_2')}\cap Y\neq\emptyset$.
    \item $(\eta_0,\eta_2,\alpha)\in (\kappa^{<\kappa})^2\times\kappa$ such that $\alpha\neq 0$ and one of the following holds:
    \begin{itemize}
        \item $\dom(\eta_0)=\dom(\eta_2)\leq \alpha$,
        \item $\dom(\eta_0)=\dom(\eta_2)+1\leq \alpha$.
    \end{itemize}
    The triples are ordered by $(\eta_0,\eta_2,\alpha)\leq (\eta'_0,\eta'_2,\beta)$ if $\eta_0\subseteq\eta'_0$, $\eta_2\subseteq\eta'_2$, and $\alpha=\beta$. 
    If $(\eta'_0,\eta'_2,\alpha)$ is not a leaf, %$(\eta'_0,\eta'_1,\alpha)< (\eta_0,\eta_1,\alpha)$, 
    then $N_{(\eta_0',\eta_2')}\cap Y\neq\emptyset$.
\end{itemize}
Finally we say that $(\emptyset,\emptyset,\emptyset)\leq (x_0,x_1,x_2)$ for all the triples, independent of which kind it is.
Let us define the labelled function, $L$, for $T$:
\begin{itemize}
    \item Case $t=(\eta_0,\eta_1,\eta_2)\in (\kappa^{<\kappa})^3$:
    \begin{itemize}
        \item If $t$ is not a leaf and $\dom(\eta_0)=\dom(\eta_2)$, then $L(t)=\cap$. If $t$ is not a leaf and $\dom(\eta_0)\neq\dom(\eta_2)$, then $L(t)=\cup$.
        \item If $t$ is a leaf and $N_{(\eta_0,\eta_1)}\cap Z=\emptyset$, then $L(t)=\kappa^\kappa$. If $t$ is a leaf and $N_{(\eta_0,\eta_1)}\cap Z\neq\emptyset$, then $L(t)=\kappa^\kappa\backslash N_{\eta_0}$.
    \end{itemize}
    \item Case $t=(\eta_0,\eta_2,\alpha)\in (\kappa^{<\kappa})^2\times\kappa$:
    \begin{itemize}
        \item If $t$ is not a leaf and $\dom(\eta_0)=\dom(\eta_2)$, then $L(t)=\cap$. If $t$ is not a leaf and $\dom(\eta_0)\neq\dom(\eta_2)$, then $L(t)=\cup$.
        \item If $t$ is a leaf and $N_{(\eta_0,\eta_2)}\cap Y=\emptyset$, then $L(t)=\kappa^\kappa\backslash N_{\eta_0}$. If $t$ is a leaf and $N_{(\eta_0,\eta_2)}\cap Y\neq\emptyset$, then $L(t)=\kappa^\kappa$.
    \end{itemize}
\end{itemize}

Let us show that $T$ has no $\kappa$-branch. Assume, for sake of contradiction, that $T$ has a $\kappa$-branch. By the construction of $T$, there is $(\eta_0,\eta_1,\eta_2)\in \kappa^\kappa$ such that for all $\beta<\kappa$, $(\eta_0\restriction\beta,\eta_1\restriction\beta,\eta_2\restriction\beta)\in T$. Thus, for all $\beta<\kappa$, $N_{(\eta_0\restriction\beta,\eta_1\restriction\beta)}\cap Z\neq\emptyset$ and $N_{(\eta_0\restriction\beta,\eta_2\restriction\beta)}\cap Y\neq\emptyset$. Since $Z$ and $Y$ are $\kappa$-closed, $(\eta_0,\eta_1)\in Z$ and $(\eta_0,\eta_2)\in Y$. So $\eta_0\in Y\cap Z$, a contradiction. 

Notice that $T$ is closed under coordinate wise increasing unions, so every increasing sequence in $T$ has a supremum. Since $|T|\leq\kappa^{<\kappa}$, $(T,L)$ is a good labelled $(\kappa,\kappa^{<\kappa})$-tree.

Let us show that for all $\xi\in \kappa^\kappa$, $\xi\in X$ if and only if $II\uparrow GB(\xi,(T,L))$. For all $\eta\in\kappa^\kappa$, let $\pi(\eta)$ be the strategy in which II chooses initial segments of $\eta$, i.e. if the game is at $(\eta'_0,\eta'_1,\eta'_2)$ such that $L((\eta'_0,\eta'_1,\eta'_2))=\cup$, II chooses the unique triple $(\eta_0,\eta_1,\eta_2)$ that satisfies $\dom(\eta_0)=\dom(\eta_1)=\dom(\eta_2)$, $\eta'_0=\eta_0$, $\eta'_1=\eta_1$, and $\eta'_2\subseteq\eta_2\subseteq \eta$ (the same idea applies if the game is at $(\eta'_0,\eta'_2,\alpha)$).

Let $\xi\in \kappa^\kappa$ be such that $\xi\in X$. So, there is $\eta\in \kappa^\kappa$ such that $(\xi,\eta)\in Y$. Suppose that II plays following $\pi(\eta)$ and the game ends at the leaf $(x_0,x_1,x_2)$. 
\begin{itemize}
    \item Case $(x_0,x_1,x_2)=(\eta_0,\eta_1,\eta_2)\in (\kappa^{<\kappa})^3$: If $\eta_0\not\subseteq\xi$, then $\xi\in \kappa^\kappa\backslash N_{\eta_0}$. Thus $\xi\in L((\eta_0,\eta_1,\eta_2))$ and it doesn't matter whether $N_{(\eta_0,\eta_1)}\cap Z=\emptyset$ or not. On the other hand, by the definition of $\pi(\eta)$, $\eta_2\subseteq \eta$. Thus, if $\eta_0\subseteq\xi$, then $N_{(\eta_0',\eta_2')}\cap Y\neq\emptyset$ and $N_{(\eta_0',\eta_1')}\cap Z=\emptyset$. So $\xi\in L((\eta_0,\eta_1,\eta_2))$.
    \item Case $(x_0,x_1,x_2)=(\eta_0,\eta_2,\alpha)\in (\kappa^{<\kappa})^2\times\kappa$: If $N_{(\eta_0,\eta_2)}\cap Y\neq\emptyset$, then $\xi\in \kappa^\kappa=L((\eta_0,\eta_2,\alpha))$. On the other hand, by the definition of $\pi(\eta)$, $\eta_2\subseteq \eta$. Thus, if $N_{(\eta_0,\eta_2)}\cap Y=\emptyset$, then $\eta_0\not\subseteq\xi$. So $\xi\in \kappa^\kappa\backslash N_{\eta_0}=L((\eta_0,\eta_1,\eta_2))$.
\end{itemize}
We conclude that $\pi(\eta)$ is a winning strategy of II in $GB(\xi,(T,L))$.

For all $\xi,\eta\in(\kappa^\kappa)^2$, let $\sigma(\xi,\eta)$ be the strategy in which I chooses initial segments of $(\xi,\eta)$. If $\xi\notin X$, then there is $\eta\in \kappa^\kappa$ such that $(\xi,\eta)\in Z$. Clearly $\sigma(\xi,\eta)$ is a winning strategy of I in $GB(\xi,(T,L))$.

We conclude that $$\xi\in X\leftrightarrow II\uparrow GB(\xi,(T,L)).$$ 
We will finish the prove by showing that $$W=\{(\xi,\eta) : \pi(\eta)\mbox{ is a winning strategy of II in }GB(\xi,(T,L))\}$$ is $\kappa$-Borel. 
From our previous argument, if $(\xi,\eta)\in Y$, then $(\xi,\eta)\in W$. Let $(\xi,\eta)\notin Y$, then there is $\alpha<\kappa$ such that $(N_{\xi\restriction \alpha}\times N_{\eta\restriction \alpha})\cap Y=\emptyset$. If I starts by choosing first $(\emptyset,\emptyset,\alpha)$, then $(\xi\restriction 1,\emptyset,\alpha)$, and then, in each turns chooses an initial segment of $\xi$, then II will lose if II follows $\pi(\eta)$. Thus $\pi(\eta)$ is not a winning strategy. So $(\xi,\eta)\notin Y$ implies $\pi(\eta)\notin W$. Therefore $Y=W$. Since $Y$ is $\kappa$-closed, $W$ is $\kappa$-Borel.

\end{cproof}

Let $X$ be a subclass of $\kappa$-$\Sigma^1_1$. We say that a set $U\subseteq\kappa^\kappa\times\kappa^\kappa$ is \emph{universal for} $X$ if for all $A\in X$, there is $\xi\in \kappa^\kappa$ such that for all $\eta\in\kappa^\kappa$, $\eta\in A$ if and only if $(\eta,\xi)\in U$.

\begin{lemma}
    For all $\alpha<\kappa^+$ there is $R_\alpha\subseteq \kappa^\kappa\times\kappa^\kappa$ which is a $\kappa$-Borel set and universal for $\kappa\text{-}\Sigma_{\alpha}$.
\end{lemma}

\begin{cproof}
    We will proceed by induction over $\alpha$. 

    {\bf $\alpha=0$:} Let $Z$ be the set of tuples $(\{(i,j)\},\eta)$ where $i,j\in \kappa$, $\dom(\eta)=i+1$, for all $k<i$, $\eta(k)=0$, and $\eta(i)=j+1$. Let $R_0=\bigcup_{z\in Z}N_z$, it is clear that $R_0$ is $\kappa$-open and universal for $\kappa\text{-}\Sigma_{0}$

    {\bf $\alpha+1$:} Let us show the case where $\alpha$ is even, the odd case is similar. For all $t\in \kappa^{<\omega}$ different from $\emptyset$, let $p(t)$ be such that $\dom(p(t))=\dom(t)-1$ and $p(t)(n)=t(n+1)$. 

    Let $\alpha$ be even such that $R_\alpha$ exists and $(T_\alpha,L_\alpha)$ is the $\kappa$-\borelstar-code for $R_\alpha$, where $T\subset \kappa^{<\omega}$. Let $\Pi:\kappa\times\kappa\rightarrow\kappa$ be a bijection such that if $\beta<\beta'$, then $\Pi(\gamma,\beta)<\Pi(\gamma,\beta')$ for all $\gamma$. 
    
    For all $\gamma<\kappa$, $X\subseteq \kappa$, $|X|\leq\kappa$, and $\eta:X\rightarrow\kappa$, let $p_\gamma(\eta)=\xi$ be such that $\dom(\xi)=\{\beta<\kappa\mid \Pi(\gamma,\beta)\in \dom (\eta)\}$, and for all $\beta\in \dom (\xi)$, $\xi(\beta)=\eta(\Pi(\gamma,\beta))$.

    For all $\gamma<\kappa$, $X\subseteq \kappa$, $|X|\leq\kappa$, and $\eta:X\rightarrow\kappa$, let $p^*_\gamma(\eta)=\xi^*$ be such that $\dom(\xi^*)=\{\Pi(\gamma,\beta)\mid \beta\in \dom (\eta)\}$, and for all $\Pi(\gamma,\beta)\in \dom (\xi^*)$, $\xi^*(\Pi(\gamma,\beta))=\eta(\beta)$.
    
    Let $T_{\alpha+1}$ is the set of all $t\in \kappa^{<\omega}$ such that $p(t)\in T_\alpha$. Let $L_{\alpha+1}$ be such that $L_{\alpha+1}(\emptyset)=\cup$; if $t\neq\emptyset$ and it is not a leaf, then $L_{\alpha+1}(t)=L_\alpha(p(t))$; and if $t$ is a leaf and $L_\alpha(p(t))=N_{(\eta,\xi)}$, then $L_{\alpha+1}(t)=N_{(\eta,p^*_\gamma(\xi))}$, where $\gamma=t(0)$. By Lemma \ref{labelled_to_Borel} it is easy to see that $(T_{\alpha+1},L_{\alpha+1})$ is a $\kappa$-\borelstar-code.
    
    Let $R_{\alpha+1}$ be the set whose $\kappa$-\borelstar-code is $(T_{\alpha+1},L_{\alpha+1})$. Then $R_{\alpha+1}$ is $\kappa$-Borel by Lemma \ref{Borel_is_Borel*}.
    \begin{claim}
        $R_{\alpha+1}$ is universal for $\kappa\text{-}\Sigma_{\alpha+1}$.
    \end{claim}
    \begin{proof}
        Let $A\in \kappa\text{-}\Sigma_{\alpha+1}$ and $A=\bigcup_{i<\kappa}A_i$ where for all $i<\kappa$, $A_i\in \kappa\text{-}\Sigma_\alpha$. By the induction hypothesis, for all $i<\kappa$, there is $\xi_i$ such that $\eta\in A_i$ if and only if $(\eta,\xi_i)\in R_\alpha$. It is easy to see that there is $\xi$ such that for all $i<\kappa$, $p_i(\xi)=\xi_i$. From the definition of $p_i$ and $p^*_i$, $\xi$ exists. Let us show that for all $\eta\in \kappa^\kappa$, $\eta\in A$ if and only if $(\eta,\xi)\in R_{\alpha+1}$. From the way $R_{\alpha+1}$ was constructed, $(\eta,\xi)\in R_{\alpha+1}$ if and only if $II\uparrow GB((\eta,\xi),(T_{\alpha+1},L_{\alpha+1}))$. Thus it is enough to show that for all $\eta\in \kappa^\kappa$, $\eta\in A$ if and only if $II\uparrow GB((\eta,\xi),(T_{\alpha+1},L_{\alpha+1}))$.

        Let $\eta$ be such that $\eta\in A$. Therefore, there is $i<\kappa$ such that $\eta\in A_i$ and $(\eta,\xi_i)\in R_\alpha$. Thus $II\uparrow GB((\eta,\xi_i),(T_{\alpha},L_{\alpha}))$, let $\sigma$ be a winning strategy of II in $GB((\eta,\xi_i),(T_{\alpha},L_{\alpha}))$. Let II plays by the following strategy in $GB((\eta,\xi),(T_{\alpha+1},L_{\alpha+1}))$: II chooses $(0,i)$ in the first move (since $L_{\alpha+1}(\emptyset)=\cup$) and afterwards, II plays by the strategy $\sigma$. By following this strategy, II makes sure that the game ends in a leaf $t$ such that $i=t(0)$, $L_{\alpha+1}(t)=N_{(\zeta_1,p^*_i(\zeta_2))}$, where $L_\alpha(p(t))=N_{(\zeta_1,\zeta_2)}$, and $(\eta,\xi_i)\in N_{(\zeta_1,\zeta_2)}$. By the way $\xi$ was defined, $(\eta,\xi)\in N_{(\zeta_1,p^*_i(\zeta_2))}$.

        Let $\eta$ be such that $II\uparrow GB((\eta,\xi),(T_{\alpha+1},L_{\alpha+1}))$. Let $\sigma$ be a winning strategy of II in $GB((\eta,\xi),(T_{\alpha+1},L_{\alpha+1}))$. Since $L_{\alpha+1}(\emptyset)=\cup$, there is a unique $\gamma<\kappa$ such that if II plays following $\sigma$, then the game ends in a leaf $t$ such that $\gamma=t(0)$, $L_{\alpha+1}(t)=N_{(\zeta_1,p^*_\gamma(\zeta_2))}$, and $(\eta,\xi)\in N_{(\zeta_1,p^*_\gamma(\zeta_2))}$. By the way $\xi$ and $L_{\alpha+1}$ were defined, $L_\alpha(p(t))=N_{(\zeta_1,\zeta_2)}$, and $(\eta,\xi_\gamma)\in N_{(\zeta_1,\zeta_2)}$. Thus $\sigma$ induces a winning strategy of II in $GB((\eta,\xi_\gamma),(T_{\alpha},L_{\alpha}))$. We conclude that $(\eta,\xi_\gamma)\in R_\alpha$ and $\eta\in A_\gamma$. Thus $\eta\in A$.
    \end{proof}

    {\bf $\alpha$ limit:} Let $\alpha$ be such that for all $\beta<\alpha$, such $R_\beta$ exists. For all $\beta<\alpha$, let $(T_\beta,L_\beta)$ be the $\kappa$-\borelstar-code for $R_\beta$, such that $T_\beta$ is a subtree of $\kappa^{<\omega}$.

    Let $T_\alpha$ be the tree of those $t\in \kappa^{<\omega}$ such that $t=\emptyset$ or $t(0)<\alpha$, and $p(t)\in T_{t(0)}$. Let $L_\alpha$ be such that $L_\alpha(\emptyset)=\cup$ and for $t\neq\emptyset$, $L_\alpha(t)=L_{t(0)}(p(t))$. It is easy to see that $(T_{\alpha},L_{\alpha})$ is a $\kappa$-\borelstar-code.
    Let $R_\alpha$ be the set whose $\kappa$-\borelstar-code is $(T_\alpha,L_\alpha)$.
    In a similar way as in the previous claim, it is possible to prove that $R_\alpha$ is universal for $\kappa$-$\Sigma_\alpha$.

\end{cproof}

As to a counterpart of the above lemma in the bounded topology, we refer to the following fact: Assuming $2^\kappa<2^{(2^{<\kappa})}$, there is no universal set for $\kappa\text{-}\Sigma_{\alpha}$ in the space $2^\kappa$, for any $\alpha<\kappa^\kappa$, as was proved in 
 \cite[Corollary 4.16]{MR3093397}.

\begin{cor}\label{proper_hierarchy}
    The sets $\kappa$-$\Sigma_\alpha$ form a proper hierarchy.
\end{cor}
\begin{cproof}
    Assume, for sake of contradiction, that there is $\alpha<\kappa^+$ such that every $\kappa$-Borel set is a $\kappa$-$\Sigma_\alpha$ set. Let $A:=\{\eta\in \kappa^\kappa\mid (\eta,\eta)\notin R_\alpha\}$. Since $R_\alpha$ is $\kappa$-Borel, $(\kappa^\kappa\times\kappa^\kappa)\ \backslash R_\alpha$ is $\kappa$-Borel. Let $f:\kappa^\kappa\rightarrow\kappa^\kappa\times\kappa^\kappa$ be the function that maps $\eta$ into $(\eta,\eta)$. Notice that $f^{-1}[N_\eta\times N_\xi]$ is either the empty set or $N_{\eta\cup\xi}$, thus $f$ is $\kappa$-continuous. Clearly $A=f^{-1}[(\kappa^\kappa\times\kappa^\kappa)\ \backslash R_\alpha]$, so $A$ is $\kappa$-Borel. Thus $A$ is $\kappa$-$\Sigma_\alpha$. On the other hand, since $R_\alpha$ is universal for $\kappa$-$\Sigma_\alpha$, there is $\xi\in\kappa^\kappa$ such that $\eta\in A$ if and only if $(\eta,\xi)\in R_\alpha$. Thus $(\xi,\xi)\in R_\alpha$ holds if and only if $\xi\in A$. So $(\xi,\xi)\in R_\alpha$ if and only if $(\xi,\xi)\notin R_\alpha$, a contradiction. 
\end{cproof}

%This corollary implies that the hierarchy from Definition \ref{Borel_hier1} is also a proper hierarchy. 
We define the $\kappa$-Borel rank of a $\kappa$-Borel set, $A$, as the least ordinal $\alpha$, such that $A\in \kappa\text{-}\Sigma_\alpha$. We denote the $\kappa$-Borel rank of $A$ by $rk_B(A)$.

\begin{thm}
    There is a $(\kappa,\kappa^+)$-Borel set $R\subseteq \kappa^\kappa\times\kappa^\kappa$ which is universal for $\kappa$-Borel sets ($\kappa$-Borel sets).
\end{thm}
\begin{cproof}
    For all $\alpha<\kappa^+$, let $R_\alpha\subseteq \kappa^\kappa\times\kappa^\kappa$ be a $\kappa$-Borel that is universal for $\kappa\text{-}\Sigma_{\alpha}$. Let $\Pi:\kappa^+\rightarrow\kappa^\kappa$ be injective and let us denote $\Pi(\alpha)$ by $\xi_\alpha$. Let us define $R^*$ by $$R^*:=\bigcup_{\alpha<\kappa^+}R_\alpha\times\{\xi_\alpha\}\subseteq \kappa^\kappa\times\kappa^\kappa\times\kappa^\kappa.$$
    Since $$R_\alpha\times\{\xi_\alpha\}=(R_\alpha\times\kappa^\kappa)\cap (\kappa^\kappa\times\kappa^\kappa\times \{\xi_\alpha\}),$$ $R_\alpha\times\{\xi_\alpha\}$ is $\kappa$-Borel, thus it is a  $(\kappa,\kappa^+)$-Borel set. Therefore, $R^*$ is a $(\kappa,\kappa^+)$-Borel set.

    It is easy to see that there is $X\subset\kappa$ and $g:X\rightarrow\kappa$ a bijection, such that for all $\alpha<\kappa^+$ there is $\bar{R}_\alpha\subseteq \kappa^\kappa\times\kappa^\kappa$, such that for all $A\in \kappa$-$\Sigma_\alpha$, $\eta\in A$ if and only if there is $\xi$ such that $(\eta,\xi)\in \bar{R}_\alpha$ and $\xi\restriction X=\xi_\alpha\circ g$. Thus $R^*$ can be coded by a $(\kappa,\kappa^+)$-Borel set universal for $\kappa$-Borel sets.
\end{cproof}

\begin{cor}
    There is a $(\kappa,\kappa^+)$-Borel set that is not $\kappa$-Borel.
\end{cor}
\begin{cproof}
    It can be proved in a similar way as in Corollary \ref{proper_hierarchy}.
\end{cproof}

\section{Orbits of uncountable models: the non-structure case}

Having established, in Sections 2 and 3 above, the basic topological properties of the Generalized Baire Space $\kappa^\kappa$, in the absence of $\kappa^{<\kappa}=\kappa$, we turn in this section to applications in model theory.

Every structure of infinite cardinality $\kappa$ in a countable vocabulary can be canonically associated with an element of $\kappa^\kappa$, or alternatively an element of $2^\kappa$, leading to close connections between model theoretic properties of structures and topological properties of subsets of $\kappa^\kappa$. To make this translation completely explicit, we make the following definition: 

\begin{comment}

\begin{defn}
    The class of \emph{weakly} $(\kappa,\lambda)$-\emph{Borel} sets is the smallest class that contains all the open sets and is closed under complements and $\lambda$-unions and $\lambda$-intersections.
\end{defn}
\end{comment}

\begin{defn}\label{struct}
Let $L=\{Q_m\mid m\in\omega\}$ be a countable relational language.
Fix a bijection $\pi$ between $\kappa^{<\omega}$ and $\kappa$. For every $\eta\in \kappa^\kappa$ define the structure $M_{\eta}$ with domain $\kappa$ as follows:
For every tuple $(a_1,a_2,\ldots , a_n)$ in $\kappa^n$ we define $$\begin{array}{lcl}
(a_1,a_2,\ldots , a_n)\in Q_m^{M_{\eta}}&\Leftrightarrow& Q_m \text{ has arity } n \text{ and }\\
&&\eta(\pi(m,a_1,a_2,\ldots,a_n))>0.\end{array}$$
\end{defn}

Let $L$ and $\pi$ be as in Definition \ref{struct}. We say that a set $A\subseteq \kappa^\kappa$ is \emph{invariant under} $(L,\pi)$, if $\eta\in A$ and $M_\eta\cong M_{\eta'}$ imply $\eta'\in A$. Canonical examples of invariant sets are  orbits of models $M$ of size $\kappa$:
$$\Orb(M)=\{\eta\in \kappa^\kappa:M\cong M_\eta\}.$$

In the realm of countable structures, every orbit is Borel and a subset of $\omega^\omega$ is invariant Borel if and only it is the set of codes of countable models of a sentence of $L_{\omega_1\omega}$  (\cite{MR200133},
\cite{MR188059}). The purpose of this and the next section is to prove a similar result for  structures of cardinality $\kappa>\omega$. We show that the answer to this question depends very much on the stability theoretic properties of the first order theory of the structure. Previous work (e.g. \cite{MR3235820}) has focused on the case $\kappa^{<\kappa}=\kappa$. Here we do not make this assumption. 

In \cite{MR4409720} the relation between Borel sets and definability in $L_{\kappa^+,\kappa}$ was studied without the assumption $\kappa^{<\kappa}=\kappa$. The authors showed that in general Borel (and invariant) is not equivalent to definability in $L_{\kappa^+,\kappa}$, both directions may fail (see \cite[Remarks 8.12 and 8.17]{MR4409720}).

The following lemma can be found in more detail in \cite{MR4409720}. %It can be proved exactly as  \cite[Theorem 24]{MR3235820}. That theorem,  as well as its proof, are due to R. Vaught, see \cite{MR3235820} and \cite{MR363912}. 
Notice that it seems as if the below lemma solves  the ``Open problem" at the end of subsection III.1 of \cite{MR3235820}. However, in the ``Open problem" our ``Borel" refers to Borel sets  in the bounded topology, which is different from $\kappa$-Borel of Lemma \ref{4.2}. The ``Open problem" was solved by Motto Ros in [\cite{MR3093397}, Corollary 4.5].

\begin{lemma}\cite[Proposition 8.13 (a)]{MR4409720}\label{4.2}
    Suppose $\kappa>\omega$ 
    %is regular 
    and $L$ and $\pi$ are as in Definition \ref{struct}. If $A\subseteq \kappa^\kappa$ is $\kappa$-Borel and invariant under $(L,\pi)$, then there is an $L_{\kappa^+\kappa}$-sentence $\varphi$ such that $A=\{\eta\in \kappa^\kappa\mid M_\eta\models \varphi\}$.
\end{lemma}

\begin{defn}
Assume $\mathcal{T}$ a first order theory in a relational countable language, we define the isomorphism relation, $\cong_T~\subseteq \kappa^\kappa\times \kappa^\kappa$, as the relation $$\{(\eta,\xi)|(\mathcal{A}_\eta\models \mathcal{T}, \mathcal{A}_\xi\models \mathcal{T}, \mathcal{A}_\eta\cong \mathcal{A}_\xi)\text{ or } (\mathcal{A}_\eta\not\models \mathcal{T}, \mathcal{A}_\xi\not\models \mathcal{T})\}$$
\end{defn}

\begin{thm}\label{iso_borel}
Suppose $\mathcal{T}$ is a countable complete first order theory.
    Suppose $\kappa$ is regular and for all $\alpha<\kappa^+$ there are $\eta,\eta'\in \kappa^\kappa$ such that $M_\eta\models \mathcal{T}$, $M_{\eta'}\models \mathcal{T}$, $M_\eta\equiv^\alpha_{\kappa^+\kappa} M_{\eta'}$, and $M_\eta\not\cong M_{\eta'}$.
    %$\eta\in \kappa^\kappa$ are such that $M_\eta\models T$ and for all $\alpha<\kappa^+$ there is $\mathcal{A}$, such that $\mathcal{A}\not\cong M_\eta$ ($dom(\mathcal{A})=\kappa$) and $\mathcal{A}\equiv^\alpha_{\kappa^+\kappa} M_\eta$.
    Then $\cong_{\mathcal{T}}$ is not $\kappa$-Borel.%$Orb(M_\eta)$ is not $\kappa$-Borel. Hence  
\end{thm}

\begin{cproof} 

Before we start with the proof, let us make some preparations.

Let $\mathcal{T}$ be a countable complete first order theory in a relational countable language, $L$, and let $P$ be an unary relation symbol not in $L$. 
We will work with two different kind of structures, we will work with $L$-structures and with $L\cup\{P\}$-structures. Let $L=\{Q_m\mid m\in\omega\}$ and let us enumerate $L\cup\{P\}$ by $=\{P_m\mid m\in\omega\}$, where $P=P_0$ and $Q_m=P_{m+1}$. We will use the coding from Definition \ref{struct} to code $L$-structures and $L\cup\{P\}$-structures. To avoid misunderstandings, we will denote $L$-structures different from $L\cup \{P\}$-structures. Let $\eta\in \kappa^\kappa$, we will denote by $M_\eta$ the $L$-structure coded by $\eta$, and by $M^\eta$ the $L\cup \{P\}$-structure coded by $\eta$. Recall $\pi$, the bijection from Definition \ref{struct}.
%Let $\rho:\kappa\cup\kappa^{<\omega}\rightarrow\kappa$ be a bijection, we make a distinction between $\kappa$ and $\kappa^{\{0\}}$. Let us code structures differently from Definition \ref{struct}. For all $\eta\in \kappa^\kappa$, we define the structure $M^\eta$ with domain $\kappa$ in the vocabulary $L\cup \{P\}$ as follows:
%For every tuple $(a_1,a_2,\ldots , a_n)$ in $\kappa^n$ $$(a_1,a_2,\ldots , a_n)\in Q_m^{M^{\eta}}\Leftrightarrow Q_m \text{ has arity } n \text{ and }\eta(\rho(m,a_1,a_2,\ldots,a_n))>0,$$ and for all $a\in \kappa$, $$a\in P^{M^{\eta}}\Leftrightarrow \eta(\rho(a))>0.$$

Let $B=\pi[\{0\}\times \kappa]$ and fix $S_1,S_2\subseteq B$ disjoint sets of size $\kappa$ such that $B=S_1\cup S_2$. For all $\alpha\in B$, we will denote by $\rho(\alpha)$ the projection on the second coordinate of $\pi^{-1}(\alpha)$ (i.e. $\rho(\alpha)=pr_2(\pi^{-1}(\alpha))$).

It is easy to see that $\{\eta\in \kappa^\kappa\mid P^{M^{\eta}}=\rho[S_1]\}$ is $\kappa$-Borel, since it is the $\kappa$-intersection of the complement of basic $\kappa$-open sets. At the same time, the set of all $\eta\in \kappa^\kappa$ such that for all $0<m<\omega$ and $n$-tuple $(a_1,\ldots, a_n)\in \kappa^n\backslash (\rho[S_1]^n \cup \rho[S_2]^n)$, $\eta(\pi(m,a_1\ldots,a_n))=0$, is $\kappa$-Borel. Thus the intersection of these two sets is $\kappa$-Borel, let us detone it by $S$.

Let us define $h:S\rightarrow 2^\kappa\times 2^\kappa$ as follows. Let $r_1:\kappa\rightarrow \rho[S_1]$ and $r_2:\kappa\rightarrow \rho[S_2]$ be the order preserving bijection. %Notice that  $P^{M^{\xi}}=\rho^{-1}[S_1]$ and $\kappa\backslash P^{M^{\xi}}=\rho^{-1}[S_2]$, so $r_1$ and $r_2$ are independent of $\xi$.  

For all pairs $(\eta_1,\eta_2)\in 2^\kappa\times 2^\kappa$ define $h^{-1}[(\eta_1,\eta_2)]$ as the subset of $\kappa^\kappa$ of $\xi$'s such that the following hold:
\begin{itemize}
    \item If $\pi(0,a)\in S_1$, then $\xi(\pi(0,a))>0$.
    \item If $\pi(0,a)\in S_2$, then $\xi(\pi(0,a))=0$.
    \item For all $m<\omega$ and $n$-tuple $(a_1,\ldots, a_n)\in \kappa^n\backslash (\rho[S_1]^n \cup \rho[S_2]^n)$, $\xi(\pi(m,a_1\ldots,a_n))=0$.
    \item For any $n$-tuple $(a_1,\ldots, a_n)\in \kappa^n$ and $m<\omega$, $$\xi(\pi(m+1,r_1(a_1)\ldots,r_1(a_n)))>0$$ if and only if $\eta_1(\pi(m,a_1,\ldots,a_n))=1$.
    \item For any $n$-tuple $(a_1,\ldots, a_n)\in \kappa^n$ and $m<\omega$, $$\xi(\pi(m+1,r_2(a_1)\ldots,r_2(a_n)))>0$$ if and only if $\eta_2(\pi(m,a_1,\ldots,a_n))=1$.
\end{itemize}

Clearly, for all $\xi\in S$, there is a unique pair $(\eta_1,\eta_2)\in 2^\kappa\times 2^\kappa$ such that $\xi\in h^{-1}[(\eta_1,\eta_2)]$. 
Since for all $\eta_1,\eta_2\in 2^\kappa$, $h^{-1}[(\eta_1,\eta_2)]\subseteq S$, $h$ is well defined. Notice that if $h(\xi)=(\eta_1,\eta_2)$, then $r_1$ is an ismorphism from $M_{\eta_1}$ to $(M^{\xi}\cap P^{M^{\xi}})\restriction L$ and $r_2$ is an ismorphism from $M_{\eta_2}$ to $(M^{\xi}\backslash P^{M^{\xi}})\restriction L$.

%Let $\eta^\xi_1\in 2^\kappa$ be such that $r_1$ is an ismorphism from $M_{\eta^\xi_1}$ to $(M^{\xi}\cap P^{M^{\xi}})\restriction L$. Let $\eta^\xi_2\in 2^\kappa$ be such that $r_2$ is an ismorphism from $M_{\eta^\xi_2}$ to $(M^{\xi}\backslash P^{M^{\xi}})\restriction L$. Notice that since $\eta^\xi_1,\eta^\xi_2\in 2^\kappa$, $\eta^\xi_1$ and $\eta^\xi_2$ are unique. Thus, define $h(\xi)=(h_1(\xi),h_2(\xi))=(\eta_1^\xi,\eta_2^\xi)$.

\begin{claim}
    $h$ is $\kappa$-continuous in $S$, i.e. the inverse image of a $\kappa$-open, $U$, is a $\kappa$-open set intersected with $S$.
\end{claim}
\begin{cproof}
    Let $U=N_p\times N_q$ be a basic $\kappa$-open set of $2^\kappa\times2^\kappa$ and $\xi\in h^{-1}[U]$. %Recall that $P^{M^{\xi}}=\{r_1(i)\mid i<\kappa\}$ and $\kappa\backslash P^{M^{\xi}}=\{r_2(i)\mid i<\kappa\}$, and $r_1$ and $r_2$ are order preserving bijections. 
    Let $$A_p=\{\pi(m+1,r_1(a_1)\ldots,r_1(a_n))\mid (m,a_1,\ldots,a_n)\in \pi^{-1}[\dom(p)]\}$$ and $$A_q=\{\pi(m+1,r_2(a_1)\ldots,r_2(a_n))\mid (m,a_1,\ldots,a_n)\in \pi^{-1}[\dom(q)]\}.$$
    Since $U$ is a basic open set, $|\dom(p)|,|\dom(q)|<\kappa$. Thus $|A_p\cup A_q|<\kappa$. Therefore, $N_{\xi\restriction A_p\cup A_q}\cap S\subseteq h^{-1}[U]$, so $h$ is $\kappa$-continuous in $S$.
\end{cproof}

Since $S$ is $\kappa$-Borel, a variations of Fact \ref{funcion_borel} applies to $h$: If $Y\subseteq 2^\kappa\times 2^\kappa$ is $\kappa$-Borel, then $h^{-1}[Y]$ is $\kappa$-Borel.

Notice that $h$ is surjective. We will denote $h(\xi)$ by $(h_1(\xi),h_2(\xi))$.

Let us assume, for sake of contradiction, that $\mathcal{T}$ satisfies the assumptions and $\cong_T$ is $\kappa$-Borel.
Let $$X=\{\xi\in S\mid M_{h_1(\xi)}\cong M_{h_2(\xi)}\}=h^{-1}[\cong_T\cap(2^\kappa\times 2^\kappa)],$$ since $2^\kappa\times 2^\kappa$ and $\cong_T$ are $\kappa$-Borel, $X$ is $\kappa$-Borel. 
Notice that if $\eta$ and $\xi$ are such that $M^\eta$ and $M^\xi$ are isomorphic, then $(M^{\eta}\cap P^{M^{\eta}})\restriction L$ and $(M^{\xi}\cap P^{M^{\xi}})\restriction L$ are isomorphic, and the same holds for $(M^{\eta}\backslash P^{M^{\eta}})\restriction L$ and $(M^{\xi}\backslash P^{M^{\xi}})\restriction L$. 
Thus $X$ is invariant under $(L\cup\{P\},\pi)$ and by Lemma \ref{4.2}, there is an $L_{\kappa^+\kappa}$-sentence $\varphi$ such that $X=\{\eta\in \kappa^\kappa\mid M^\eta\models \varphi\}$. Let $\alpha$ be the rank of $\varphi$.

By the way $h$ was defined and our assumptions on $\mathcal{T}$, there are $\eta_0,\eta_1\in 2^\kappa$ such that $M_{\eta_0}\models \mathcal{T}$, $M_{\eta_1}\models \mathcal{T}$, $M_{\eta_0}\not\cong M_{\eta_1}$, and $M_{\eta_0}\equiv_{\kappa^+\kappa}^\alpha M_{\eta_1}$. Let $\eta,\xi\in S$ be such that $h(\eta)=(\eta_0,\eta_0)$ and $h(\xi)=(\eta_0,\eta_1)$. Clearly $\eta\in X$, $\xi\notin X$, and $M_{\eta}\equiv_{\kappa^+\kappa}^\alpha M_{\xi}$. This contradicts that $\varphi$ has rank $\alpha$.

    \end{cproof}

\begin{thm}Suppose $\mathcal{T}$ is a countable complete first order theory.
    Suppose $\kappa$ 
    is regular and $\eta\in \kappa^\kappa$ are such that $M_\eta\models \mathcal{T}$, and for all $\alpha<\kappa^+$ there is $\xi\in \kappa^\kappa$ such that $M_\xi\models \mathcal{T}$, $M_\xi \not\cong M_\eta$ and $M_\xi\equiv^\alpha_{\kappa^+\kappa} M_\eta$. Then $\Orb(M_\eta)$ is not $\kappa$-Borel. %Hence $\cong_T$ is not $\kappa$-Borel.
\end{thm}

\begin{cproof}
We will use the same preparation as is Theorem \ref{iso_borel}, i.e. the function $h$ and the models $M^\eta$.

Let us assume, for sake of contradiction, that $\mathcal{T}$ and $\eta$ satisfy the assumptions and $\Orb(M_{\eta})$ is $\kappa$-Borel. Let $$\begin{array}{lcl}
X&=&\{\xi\in S\mid M_{h_1(\xi)}\cong M_{h_2(\xi)}\cong M_{\eta}\}\\
&=&h^{-1}[(\Orb(M_{\eta})\times \Orb(M_{\eta}))\cap (2^\kappa\times 2^\kappa)],\end{array}$$ since $\Orb(M_{\eta})$ are $X$ are $\kappa$-Borel. Notice that $X$ is invariant under $(L\cup\{P\},\pi)$ and by Lemma \ref{4.2}, there is an $L_{\kappa^+\kappa}$-sentence $\varphi$ such that $X=\{\eta\in \kappa^\kappa\mid M^\eta\models \varphi\}$. Let $\alpha$ be the rank of $\varphi$.

By the way $h$ was defined and our assumptions on $\mathcal{T}$ and $\eta$, there is $\xi\in 2^\kappa$ such that $M_{\xi}\models \mathcal{T}$, $M_{\eta}\not\cong M_{\xi}$, and $M_{\eta}\equiv_{\kappa^+\kappa}^\alpha M_{\xi}$.
Let $\zeta,\zeta'\in S$ be such that $h(\zeta)=(\eta,\eta)$ and $h(\zeta)=(\eta,\xi)$. Clearly $\zeta\in X$, $\zeta'\notin X$, and $M_{\zeta}\equiv_{\kappa^+\kappa}^\alpha M_{\zeta'}$. This contradicts that $\varphi$ has rank $\alpha$.

\end{cproof}

\begin{comment}
\begin{fact}
    Let $E$ be a $\kappa$-Borel equivalence relation. Then the equivalence classes of $E$ are $\kappa$-Borel.
\end{fact}
\begin{proof}
    Let $E$ be a $\kappa$-Borel equivalence relation and let us fix $\xi\in\kappa^\kappa$. We will show that $[\xi]_E$  is $\kappa$-Borel. Let us define $f:\kappa^\kappa\rightarrow\kappa^\kappa\times\kappa^\kappa$ as $f(\eta)=(\eta,\xi)$. It is clear that $f$ is $\kappa$-continuous. On the other hand $[\xi]_E=f^{-1}[(\kappa^\kappa\times\{\xi\})\cap E]$. Since $\kappa^\kappa\times\{\xi\}$ and $E$ are $\kappa$-Borel, by Fact \ref{funcion_borel} $f^{-1}[(\kappa^\kappa\times\{\xi\})\cap E]$ is $\kappa$-Borel.
\end{proof}

\begin{cor}Suppose $T$ is a countable complete first order theory.
    Suppose $\kappa$ is regular and $\eta\in \kappa^\kappa$ are such that $M_\eta\models T$ and for all $\alpha<\kappa^+$ there is $\mathcal{A}$, such that $\mathcal{A}\not\cong M_\eta$ ($dom(\mathcal{A})=\kappa$) and $\mathcal{A}\equiv^\alpha_{\kappa^+\kappa} M_\eta$. Then there is $\xi\in \kappa^\kappa$ such that $\Orb(M_\xi)$ is not $\kappa$-Borel.
\end{cor}
\end{comment}

\begin{cor}
    Suppose $\mathcal{T}$ is a countable non-classifiable theory (i.e. $\mathcal{T}$ is not a superstable stable theory without DOP and without OTOP) and $\kappa>\omega$ is regular, then the isomorphism relation of $\mathcal{T}$, $\cong_\mathcal{T}$, is not $\kappa$-Borel on the Generalized Baire Space $\kappa^\kappa$.
\end{cor}

\begin{cproof}
    From \cite{MR899084} Theorem 0.2 (Main Conclusion), there are $2^\kappa$ pairwise non-isomorphic $L_{\infty \kappa}$-equivalent models of power $\kappa$. The result follows from Theorem \ref{iso_borel}.
\end{cproof}

\section{Orbits of uncountable models: the structure case}

We prove Borel-definability results for orbits and for the isomorphism relation of models of cardinality $\kappa$  in the structure case of Shelah's Main Gap. 
Some of the results in this section have as an assumption the regularity of $\kappa$ and $\kappa=\kappa^\omega$. This assumption is, of course, weaker than the assumption $\kappa=\kappa^{<\kappa}$ prevalent in earlier work in this area.

%For all $n<\omega$, let $<^n$ be a well-ordering of $\kappa^n$ of order type $\kappa$. For all $n<\omega$ and cardinals $\theta$, let $R_\theta^n$ be the set of tuples $a\in \kappa^n$ such that $|\{b\in\kappa^n\mid b<^n a\}|=\theta$. 
Let $L^+=\{<^n, R_\alpha\mid 0<n<\omega, \alpha<\kappa\}$ and $M^+$ and $L^+$-structure such that $\dom (M^+)=\kappa$ and for all $0<n<\omega$, $<^n {}^{M^+}$ is a well-ordering of $\kappa^n$ of order type $\kappa$ and $R_\alpha^{M^+}=\alpha$. We assume that $L^+\cap L=\emptyset$ and write $L^*=L\cup L^+$. By $L^*_{\kappa^+\lambda}$ we mean $L_{\kappa^+\lambda}$-formulas in the vocabulary $L^*$. By $M_\eta^*$ we mean the $L^*$-structure such that $M_\eta^*\restriction L=M_\eta$ and $M_\eta^*\restriction L^+=M^+$. Notice that instead of these $L^+$ and $M^+$, we could use any $L^+$ of size at most $\kappa$ and an $M^+$ structure such that $\dom (M^+)=\kappa$ and Lemma \ref{5.1}, below, still holds. However, these $L^+$ and $M^+$ are the ones that we will use in this section.
%Let $L^+=\{<^n, R^n_\theta\mid n<\omega, \theta<\kappa\}$ and $L^*=L\cup L^+$. We use $L^*_{\kappa^+\lambda}$ to denote the set of sentences of $L_{\kappa^+\lambda}$ when the vocabulary is $L^*$.
%We denote by $M_\eta^*$ the expansion of $M_\eta$ to $L^*$ with the right interpretation of the elements of $L^+$ (e.g. $<^n$ is a well-ordering of $\kappa^n$ of order type $\kappa$).

The following Lemma extends [\cite{MR4409720}, Corollary 8.10 (a)] to the models $M^*_\eta$.

\begin{lemma}\label{5.1}
%    Assume $\kappa$ is such that $\kappa^{<\lambda}\leq\kappa$. 
    Let $X\subseteq \kappa^\kappa$. If there is an $L^*_{\kappa^+\lambda}$-sentence $\varphi$  such that $\eta\in X$ if and only if $M_\eta^*\models \varphi$, whenever $\eta\in\kappa^\kappa$, then $X$ is $(\kappa,\kappa^{< \lambda})$-Borel.
\end{lemma}

\begin{cproof}
    Let $\varphi$ be an $L^*_{\kappa^+\lambda}$-formula with variables $\bar{x}$, let $\gamma<\lambda$ be the arity of $\bar{x}$. It is enough to show that for any free variables assignment $s:\gamma\rightarrow \kappa$, the set $$X^s_\varphi:=\{\eta\in \kappa^\kappa\mid M_\eta^*\models_s \varphi\}$$ is $(\kappa,\kappa^{< \lambda})$-Borel. 
    Let $s$ be a free variables assignment, we will proceed by induction on formulas, $\psi$, to show that $X^s_\psi$ is $(\kappa,\kappa^{< \lambda})$-Borel.

    Case $\varphi$ is atomic: 
    \begin{itemize}
        \item $\psi$ is an $L$-formula: Let $\psi=R(x_1,x_2,\ldots , x_n)$, thus $M_\eta^*\models_s \psi$ if and only if $\eta(\alpha)>0$, where $\alpha=\pi(k,s(1),s(2),\ldots , s(n))$ and $R=Q_k$. Therefore, $X^s_\psi$ is the union of $\kappa$ basic $\kappa$-open sets. So, it is $(\kappa,\kappa^{< \lambda})$-Borel. 
        \item $\psi$ is an $L^+$-formula: Let $\psi=P(x_1,x_2,\ldots , x_n)$. Notice that for all $N,M\models \psi$, $N\restriction L^+= M\restriction L^+$. Thus $(s(1),s(2),\ldots , s(n))\in P^M$ if and only if $(s(1),s(2),\ldots , s(n))\in P^N$. Therefore $X^s_\psi=\emptyset$ or $X^s_\psi=\kappa^\kappa$.
    \end{itemize}
    
    Case negation: Let $\psi=\neg \psi'$ such that $\psi'$ is a $L^*$-formula and $X^s_{\psi'}$ is $(\kappa,\kappa^{< \lambda})$-Borel. Since $(\kappa,\kappa^{< \lambda})$-Borel is closed under complement, $X^s_\psi$ is $(\kappa,\kappa^{< \lambda})$-Borel. 
    
    Case conjunction and disjunction: Let $\{\psi_i\mid i<\delta<\kappa^{< \lambda}\}$ be a set of $L^*$-formulas such that for all $X^s_{\psi_i}$ is $(\kappa,\kappa^{< \lambda})$-Borel. Let $\psi=\bigwedge_{i<\delta}\psi_i$, so $X^s_\psi=\bigcap_{i<\delta}X^s_{\psi_i}$. Thus $X^s_\psi$ is $(\kappa,\kappa^{< \lambda})$-Borel. The disjunction case is similar.

    Case quantifiers: Let $\{V_i\mid i<\delta<\kappa^{< \lambda}\}$ be a set of variables and $\psi'$ a $L^*$-formula such that $X^s_{\psi'}$ is $(\kappa,\kappa^{< \lambda})$-Borel. Let $\psi=\exists V_0 \exists V_1 \cdots \psi'$, it is clear that $X^s_\psi=X^s_{\exists\bar{V} \psi'}=\bigcup_{\bar a\in \kappa}X^{s(\bar a /\bar V)}_{\psi'}$, thus $X^s_\psi$ is $(\kappa,\kappa^{< \lambda})$-Borel. The universal quantifier case is similar.
\end{cproof}

We call $\mathbf{P}\subseteq\mathcal{P}(\kappa^n)$ a property of $n$-ary relations. We say that a relation $R\subseteq \kappa^n$ has  property $\mathbf{P}$  if $R\in \mathbf{P}$.
The previous lemma can be generalized to any  property $\mathbf{P}$ of $n$-ary relations, the proof is similar.

\begin{fact}
    To show that a property $\mathbf{P}$ of $n$-ary relations is $\kappa$-Borel, it is enough to find $L$, $\mathcal{A}$, and $\varphi$ such that.
    \begin{enumerate}
        \item $L$ is a signature;
        \item $\mathcal{A}$ is an $L$-structure with domain $\kappa$;
        \item $\varphi\in L_{\kappa^+\omega}$ with the signature $L\cup \{R\}$
        \item $\forall R\subseteq \kappa^2$, $R$ has the property $\mathbf{P}$ if and only if $(\mathcal{A},R)\models \varphi$.
    \end{enumerate}
\end{fact}

The following surprising lemma is in sharp contrast to the famous result of classical Descriptive Set Theory to the effect that the set of elements of $\omega^\omega$ that code a well-order is $\Pi^1_1$-complete and therefore not Borel.
\begin{lemma}
    If $\mbox{\rm cf}(\kappa)>\omega$, then the property of $R\subseteq \kappa\times\kappa$ being a well-ordering is $\kappa$-Borel.
\end{lemma}
\begin{cproof} We use Fact 5.2.
    Let $L=\{(S_\alpha)_{\alpha<\kappa}\}$, $\mathcal{A}=(\kappa,(S_\alpha)_{\alpha<\kappa})$, where the interpretation of $S_\alpha$ is $\{\beta<\kappa\mid \beta<\alpha\}$. Let $R$ be a relational symbol interpreted as a linear order. For all $\alpha,\beta<\kappa$ and $i<\omega$, we define $\varphi^\alpha_\beta(x_i)$ in a recursive way as follows:
    \begin{itemize}
        \item $\varphi^\alpha_0(x_i):= ``x_i=x_i"$,
        \item if $\beta>0$, then $\varphi^\alpha_\beta(x_i):=\bigwedge_{\gamma_{i+1}<\beta}\exists x_{i+1}(S_\alpha(x_{i+1})\wedge (x_{i+1}\ R\ x_i)\ \wedge\varphi^\alpha_{\gamma_{i+1}}(x_{i+1}))$.
    \end{itemize}
    Let $\varphi:=\bigvee_{\alpha<\kappa}\bigwedge_{\gamma_0<\kappa}\exists x_0(S_\alpha(x_0)\wedge \varphi^\alpha_{\gamma_0}(x_0))$.
    Notice that $(\kappa,R)$ is a well-order if and only if for all $\alpha<\kappa$, $(\alpha,R\restriction S_\alpha)$ is a well-order.
    \begin{claim}
        $(\kappa,R)$ is not a well-order if and only if $(\kappa,R,(S_\alpha)_{\alpha<\kappa})\models \varphi$.
    \end{claim}
    \begin{cproof}
        $\Rightarrow$) Suppose $(\kappa,R)$ is not a well-order. Then there is $\alpha<\kappa$ such that $(\alpha,R\restriction S_\alpha)$ is not a well-order. So there is a sequence $\langle a_i\mid i<\omega\rangle$ such that for all $i<\omega$, $a_{i+1}Ra_i$. Let us show that for all $\gamma_0<\kappa$, $II$ has a winning strategy for the  semantic game for $$(\kappa,R,(S_\beta)_{\beta<\kappa})\models \exists x_0(S_\alpha(x_0)\wedge \varphi^\alpha_{\gamma_0}(x_0)).$$ 
        In the $i$th-round, $II$ chooses $a_i$ as an interpretation for $x_i$. Since $\langle a_i\mid i<\omega\rangle$ is an infinite descending sequence, $a_i$ is a valid move and $II$ doesn't lose. Thus $(\kappa,R,(S_\beta)_{\alpha<\kappa})\models \varphi$.

        $\Leftarrow$) Suppose that $(\kappa,R)$ is a well-order. Let us show that for all $\alpha<\kappa$, $I$ has a winning strategy for the  semantic game for $$(\kappa,R,(S_\beta)_{\beta<\kappa})\models \bigwedge_{\gamma_0<\kappa} \exists x_0(S_\alpha(x_0)\wedge \varphi^\alpha_{\gamma_0}(x_0)).$$
        Let us describe a winning strategy for $I$. $I$ start by choosing $\gamma_0=otp(S_\alpha,R\restriction S_\alpha)$. Suppose that in the $i$th-round, $I$ has chosen $\langle\gamma_j\mid j\leq i\rangle$, $II$ has chosen $\langle a_j\mid j\leq i\rangle$, and $II$ has not lost. Then $I$ chooses $\gamma_i=otp(A_i,R\restriction A_i)$, where $A_i=\{\delta\in S_\alpha\mid \delta R a_i\}$. 
        To show that this is a winning strategy, notice that $(\alpha,R\restriction S_\alpha)$ is a well-order (since $R$ is a well-order). Thus, no  matter how $II$ plays, there is $j<\omega$ such that in the $i$th-round II has chosen   the $R$-least element $a_j$ of $\alpha$. So, $\gamma_j=\emptyset$ and $II$ cannot choose an interpretation of $x_{j+1}$ without losing. Thus $(\kappa,R,(S_\alpha)_{\alpha<\kappa})\not\models \varphi$.
    \end{cproof}
\end{cproof}

\begin{thm}\label{orbit_thm}
Let $\mathcal{T}$ be a countable $\omega$-stable NDOP shallow $L$-theory. Suppose $\kappa$ is a regular cardinal such that $\kappa^\omega=\kappa$, and let $\eta\in \kappa^\kappa$ be such that $M_\eta\models \mathcal{T}$. Then there is $\varphi\in L^*_{\kappa^+\omega_1}$ such that $M_\xi\cong M_\eta$ if and only if $M_\xi^*\models\varphi$. Hence $\Orb(M_\eta)$ is $\kappa$-Borel. 
\end{thm}

Let us make some preparations before we prove Theorem \ref{orbit_thm}. From now on in this section, we will work under the assumption that $\mathcal{T}$ is a countable $\omega$-stable NDOP shallow $L$-theory. Let us fix $\eta$ and $\mathcal{T}$ such that satisfy the assumptions of Theorem \ref{orbit_thm}, thus $M_\eta\models \mathcal{T}$. 

Suppose $\mathcal{B}$ and $\mathcal{A}$ are countable such that $\mathcal{B}\preceq\mathcal{A}\preceq M$, where $M$ is a model of $\mathcal{T}$ (not necessary $M_\eta$). 
We will write $tp(a,\mathcal{A})\dashv\mathcal{B}$ for models $\mathcal{A}$ and $\mathcal{B}$, when $tp(a,\mathcal{A})$ is orthogonal to $\mathcal{B}$. We will denote by $\mathcal{B}[a]$ the primary model over $\mathcal{B}a$.

Given a tree $T$, for all $t\in T$ we define $ht_T(t)$ as the order type of $\{u\in T\mid u< t\}$.
 We say that $T'\subseteq T$ is a subtree of $T$ if $T'$ is a substructure of $T$, $T'$ has a root, and for all $t<t'<t''$ in $T$, if $t,t''\in T'$, then $t'\in T'$. 

 Recall the rank of a tree $T$ without infinite branches, $rk(T)$ from Definition \ref{rank_tree}.

\begin{defn}\label{good_labelled_tree_empt}
    We say that $(T,a_t,\mathcal{A}_t)_{t\in T}$ is a good labelled tree of $M$ (we drop $M$ when it is clear from the context) if the following holds:
    \begin{enumerate}
        \item $T$ is a subtree of $(\kappa^+)^{<\omega}$.
        \item If $t^\frown i\in T$, and $j<i$, then $t^\frown j\in T$.
        \item If $ht_T(t)=0$, then $\mathcal{A}_{t}\prec M$ is primary over the empty set and $a_t\in\mathcal{A}_{t}$.
        %\item If $t=\{(0,0)\}$, then $a_t=a_\emptyset=\emptyset$. 
        \item If $t\in T$ and $ht_T(t)\ge 1$, then $\mathcal{A}_t =\mathcal{A}_{t^-}[a_t]\preceq M$ and $a_t\not \subseteq\mathcal{A}_{t^-}$.
       % \item If $t\in T$ and $ht_T(t)\ge 1$, then $tp(a_t / \mathcal{A}_{t^-})$ is non-algebraic.
        \item If $t\in T$ and $ht_T(t)\ge 2$, then $tp(a_t / \mathcal{A}_{t^-})\dashv \mathcal{A}_{t^{--}}$.
        \item If $t^\frown i\in T$, then $$a_{t^\frown i}\downarrow_{\mathcal{A}_t} \bigcup_{j<i}a_{t^\frown j}.$$
    \end{enumerate}
    We say that $(T,a_t,\mathcal{A}_t)_{t\in T}$ is a good labelled tree over a countable $\mathcal{A}\prec M$ of $M$, if $(T,a_t,\mathcal{A}_t)_{t\in T}$ satisfies (1), (2), (4), (5), and (6), and the following:
\begin{itemize}
    \item [($3'$)] If $ht_T(t)=0$, then $\mathcal{A}_{t}=\mathcal{A}$ and $a_t\in \mathcal{A}$.
\end{itemize}
\end{defn}

\begin{defn}\label{good_labelled_tree}
    We say that $(T,a_t,\mathcal{A}_t)_{t\in T}$ is a good labelled tree over $(\mathcal{B},\mathcal{A})$, $\mathcal{B}\prec\mathcal{A}\prec M$ countable, of $M$ if the following holds:
    \begin{enumerate}
        \item $T$ is a subtree of $(\kappa^+)^{<\omega}$.
        \item If $t^\frown i\in T$, $ht_T(t)\ge 1$, and $j<i$, then $t^\frown j\in T$.
        \item If $ht_T(t)=0$, then $\mathcal{A}_{t}=\mathcal{B}$, $t$ has a unique immediate successor, $t^+$, $\mathcal{A}_{t^+}=\mathcal{A}$, $a_t\in \mathcal{B}$ and $a_{t^+}\in \mathcal{A}$.
        %\item If $t=\{(0,0)\}$, then $a_t=a_\emptyset=\emptyset$. 
        \item If $t\in T$ and $ht_T(t)\ge 1$, then $\mathcal{A}_t =\mathcal{A}_{t^-}[a_t]\preceq M$ and $a_t\not \subseteq\mathcal{A}_{t^-}$.
       % \item If $t\in T$ and $ht_T(t)\ge 1$, then $tp(a_t / \mathcal{A}_{t^-})$ is non-algebraic.
        \item If $t\in T$ and $ht_T(t)\ge 2$, then $tp(a_t / \mathcal{A}_{t^-})\dashv \mathcal{A}_{t^{--}}$.% and $tp(a_t / \mathcal{A}_{t^-})$ is non-algebraic.
        \item If $t^\frown i\in T$, then $$a_{t^\frown i}\downarrow_{\mathcal{A}_t} \bigcup_{j<i}a_{t^\frown j}.$$
    \end{enumerate}
\end{defn}

Notice that since $\mathcal{A}_t =\mathcal{A}_{t^-}[a_t]$, $a_t\triangleright_{\mathcal{A}_{t^-}}\mathcal{A}_t$.

\begin{defn}
        A good labelled tree $(T, a_t, \mathcal{A}_t)_{t\in T}$ is maximal if for all $t\in T$, there is no $a\in M$ such that $a\downarrow_{\mathcal{A}_t}\bigcup\{\mathcal{A}_{t^\frown i}\mid t^\frown i\in T\}$, $tp(a/\mathcal{A}_{t})$ is non-algebraic, and in case $ht_T(t)>0$, $tp(a/\mathcal{A}_{t})\dashv\mathcal{A}_{t^{-}}$ (i.e. there is no proper extension of the good labelled tree). A maximal good labelled tree over $\mathcal{A}$, or over $(\mathcal{B},\mathcal{A})$, is defined in a similar way, following the definition above.

\end{defn}

\begin{defn}
    Suppose $(T,a_t,\mathcal{A}_t)_{t\in T}$ is a good labelled tree and $(T',a'_t,\mathcal{A}_t')_{t\in T'}$ is a good labelled tree. Then $(f,g)$ is an isomorphism  from $(T,a_t,\mathcal{A}_t)_{t\in T}$ to $(T',a'_t,\mathcal{A}_t')_{t\in T'}$ if $f$ is an isomorphism from $(T,<)$ to $(T',<')$ and $g:\cup_{t\in T}\mathcal{A}_t\rightarrow\cup_{t\in T'}\mathcal{A}'_t$ is such that for all $t\in T$, $g\restriction \mathcal{A}_t$ is an isomorphism from $\mathcal{A}_t$ to $\mathcal{A}'_{f(t)}$ and for all $t\in T$, $ht_T(t)\ge 1$, $g(a_t)=a'_{f(t)}$.

    %We say that $(f,g)$ is a strong isomorphism if it is an isomorphism and also $g(a_t)=a'_{f(t)}$ for all $t\in T$ such that $ht_T(t)=0$.
\end{defn}

If $(T,a_t,\mathcal{A}_t)_{t\in T}$ and $(T',a'_t,\mathcal{A}_t')_{t\in T'}$ are good labelled trees over $\mathcal{A}_r$, $r$ the root of $T$ and $T'$ ($T$ and $T'$ have the same root). Then we say that they are isomorphic over $\mathcal{A}_r$ if there is an isomorphism $(f,g)$ such that $g\restriction \mathcal{A}_r=id_{\mathcal{A}_r}$.

%We define isomorphisms and strong isomorphisms between good labelled trees over $\mathcal{A}$ (or over $(\mathcal{B},\mathcal{A})$) in a similar way.

\begin{fact}
    If $b\in \mathcal{A}[a]$ and $tp(a/\mathcal{A})\dashv \mathcal{B}$, then $tp(b/\mathcal{A})\dashv\mathcal{B}$. 
\end{fact}
\begin{cproof}
    Let us assume, for sake of contradiction, that $b\in \mathcal{A}[a]$, $tp(a/\mathcal{A})\dashv \mathcal{B}$, and $tp(b/\mathcal{A})$ is not orthogonal to $\mathcal{B}$. Therefore, there are $c$ and $d$ such that $c\downarrow_{\mathcal{A}}b$, $d\downarrow_{\mathcal{B}}\mathcal{A}c$, and $b\not\downarrow_{\mathcal{A}c}d$. Without loss of generality, we can choose $c$ such that $c\downarrow_{\mathcal{A}b}a$. Since $c\downarrow_\mathcal{A} b$, by transitivity we conclude $c\downarrow_\mathcal{A} a$. On the other hand $tp(a/\mathcal{A})\dashv\mathcal{B}$ and $d\downarrow_{\mathcal{B}}\mathcal{A}c$, thus $a\downarrow_{\mathcal{A}c}d$. Therefore, $a\downarrow_{\mathcal{A}}cd$. Since $a\triangleright_\mathcal{A}b$, $b\downarrow_{\mathcal{A}}cd$. This implies $b\downarrow_{\mathcal{A}c}d$, a contradiction.
\end{cproof}

For all $t\in T$ we will denote by $T_{\ge t}$ the subtree $\{a\in T\mid a\ge t\}$. Let $t\in T$,
\begin{itemize}
    \item If $t^-$ exists, then we define $T_t$ as the subtree $T_{\ge t}\cup \{t^-\}$;
    \item If $t^-$ does not exist, then we define $T_t$ as the subtree $T_{\ge t}$. 
\end{itemize}

\begin{defn}\label{QD_over}
    We say that a good labelled tree $(T,a_t,\mathcal{A}_t)_{t\in T}$ of $M$ is a QD (quasi-decomposition) of $M$ (we drop $M$ when it is clear from the context) if the following holds:
    \begin{enumerate}
        \item If $t^\frown i\in T$, then $MR(tp(a_{t^\frown i}/\mathcal{A}_{t}))$ is the least possible, i.e. there is no $a\in M$ such that it satisfies that $a\not\subseteq \mathcal{A}_t$, Definition \ref{good_labelled_tree_empt} (5) and (6), and $MR(tp(a/\mathcal{A}_{t}))<MR(tp(a_{t^\frown i}/\mathcal{A}_{t}))$.
%        If $t\in T$ and $ht_T(t)\ge 1$, then $MR(tp(a_t/\mathcal{A}_{t^-}))$ is the least possible, i.e. there is no $a'_t\in M$ such that it satisfies Definition \ref{good_labelled_tree_empt} (5), (6), and (7), and $MR(tp(a'_t/\mathcal{A}_{t^-}))<MR(tp(a_t/\mathcal{A}_{t^-}))$. 
        For all $t\in T$ such that $ht_T(a_t)\ge 1$, we write $p_t$ for $tp(a_t/\mathcal{A}_{t^-})$.
        \item If $t^\frown i\in T$, $k<j<i$, and $p_{t^\frown k}=p_{t^\frown i}$, then $p_{t^\frown j}= p_{t^\frown i}$.
        
        If $t^\frown i\in T$ and for all $j<i$, $p_{t^\frown i}\neq p_{t^\frown j}$, then for all $j<i$, there is no $a\in M$ such that $tp(a/\mathcal{A}_t)=p_j$ and $$a\downarrow_{\mathcal{A}_t} \bigcup_{\beta<i}a_{t^\frown \beta}.$$
  %       If $t^\frown i\in T$ is such that for all $j<i$, $p_{t^\frown i}\neq p_{t^\frown j}$ and $k$ is the least such that for all $k< j<i$, $p_{t^\frown k}= p_{t^\frown j}$, then there is no $a$ such that $tp(a/\mathcal{A}_{t})=p_{t^\frown k}$ and $$a\downarrow_{\mathcal{A}_t} \bigcup_{\beta<i}a_{t^\frown \beta}.$$
         
        \item If $t^\frown i\in T$, $k<j<i$, and $(T_{t^\frown i}, a_u, \mathcal{A}_u)_{u\in T_{t^\frown i}}$ is a maximal good labelled tree over $(\mathcal{A}_t,\mathcal{A}_{t^\frown i})$ and $(T_{t^\frown i}, a_u, \mathcal{A}_u)_{u\in T_{t^\frown i}}$ is isomorphic to $(T_{t^\frown k}, a_u, \mathcal{A}_u)_{u\in T_{t^\frown k}}$ over $\mathcal{A}_t$, then $(T_{t^\frown j}, a_u, \mathcal{A}_u)_{u\in T_{t^\frown j}}$ is isomorphic to $(T_{t^\frown i}, a_u, \mathcal{A}_u)_{u\in T_{t^\frown i}}$ over $\mathcal{A}_t$.
        
%         strongly isomorphic to $(T_{t^\frown k}, a_u, \mathcal{A}_u)_{u\in T_{t^\frown k}}$, then $(T_{t^\frown i}, a_u, \mathcal{A}_u)_{u\in T_{t^\frown i}}$ is strongly isomorphic to $(T_{t^\frown j}, a_u, \mathcal{A}_u)_{u\in T_{t^\frown j}}$. %Where $$T_{t^\frown i}=T_{\ge t^\frown j}\cup \{t,t^-\}$$ when $t^{-}$ exists, otherwise $$T_{t^\frown i}=T_{\ge t^\frown j}\cup \{t\}.$$
        
       % If there is an isomorphism $f:T\rightarrow T'$, and there is an elementary bijection $g:\bigcup_{t\in T}\mathcal{A}_t\rightarrow \bigcup_{t\in T'}\mathcal{A}'_t$ such that for all $t\in T$, $g(a_t)=a'_{f(t)}$ and $g(\mathcal{A}_t)=\mathcal{A}'_{f(t)}$. Then $(T,a_t,\mathcal{A}_t)_{t\in T}$ is isomorphic to $(T',a'_t,\mathcal{A}'_t)_{t\in T'}$.

If $t^\frown i\in T$, is such that $(T_{t^\frown i}, a_u, \mathcal{A}_u)_{u\in T_{t^\frown i}}$ is a maximal good labelled tree over $(\mathcal{A}_t,\mathcal{A}_{t^\frown i})$ and for all $j<i$, $(T_{t^\frown j}, a_u, \mathcal{A}_u)_{u\in T_{t^\frown j}}$ is not isomorphic to $(T_{t^\frown i}, a_u, \mathcal{A}_u)_{u\in T_{t^\frown i}}$ over  $\mathcal{A}_t$. Then for all $j<i$, there is no good labelled tree $(T', a'_u, \mathcal{A}'_u)_{u\in T'}$ over $(\mathcal{A}'_r, \mathcal{A}'_{r^+})$, where $r$ is the root of $T'$, such that it is a maximal labelled tree over $(\mathcal{A}'_r, \mathcal{A}'_{r^+})$, $\mathcal{A}'_r=\mathcal{A}_t$, and $(T', a'_u, \mathcal{A}'_u)_{u\in T'}$ is isomorphic to $(T_{t^\frown j}, a_u, \mathcal{A}_u)_{u\in T_{t^\frown j}}$ over $\mathcal{A}_t$ and $$a_{r^+}\downarrow_{\mathcal{A}_t} \bigcup_{k<i}a_{t^\frown k}.$$

     %   If $t^\frown i\in T$ is such that for all $j<i$, $(T_{t^\frown i}, a_u, \mathcal{A}_u)_{u\in T_{t^\frown i}}$ and $(T_{t^\frown j}, a_u, \mathcal{A}_u)_{u\in T_{t^\frown j}}$ are not strongly isomorphic, and $k$ is the least such that for all $k< j<i$, $(T_{t^\frown k}, a_u, \mathcal{A}_u)_{u\in T_{t^\frown k}}$ and $(T_{t^\frown j}, a_u, \mathcal{A}_u)_{u\in T_{t^\frown j}}$ are strongly isomorphic. Then there is no $(T',a'_u,\mathcal{A}'_u)_{u\in T}$ such that it is a maximal label tree over $(\mathcal{A}_t,\mathcal{A}'_{u^*})$, where $u^*$ is the only element of $T'$ with $ht_{T'}(u^*)=1$, $\mathcal{A}'_{u^*}=\mathcal{A}_t[a'_{u^*}]$, $(T',a'_u,\mathcal{A}'_u)_{u\in T}$ is strongly isomorphic to $(T_{t^\frown k},a_u,\mathcal{A}_u)_{u\in T_{t^\frown k}}$, and $a_{u^*}\downarrow_{\mathcal{A}_t}\bigcup_{j<i}a_{t^\frown j}$
        
        %hen there is no $T'$ such that for all $j<i$, $(T_{t^\frown j}, a_u, \mathcal{A}_u)_{u\in T_{t^\frown j}}\neq (T',a'_t,\mathcal{A}'_t)_{t\in T'}$, and $(T_{t^\frown k}, a_u, \mathcal{A}_u)_{u\in T_{t^\frown k}}$ and $(T',a'_t,\mathcal{A}'_t)_{t\in T'}$ are isomorphic.

        \item If $t^\frown i\in T$, $j<i$, then $(T_{t^\frown j},a_u,\mathcal{A}_u)_{u\in T_{t^\frown j}}$ is a maximal labelled tree over $(\mathcal{A}_t,\mathcal{A}_{t^\frown j})$.
    \end{enumerate}
\end{defn}

We define a QD over $\mathcal{A}$ and QD over $(\mathcal{B},\mathcal{A})$, in a similar way.

We  denote by $Y^i_t$ the set of those $j<\kappa^+$ such that $(T_{t^\frown i}, a_u, \mathcal{A}_u)_{u\in T_{t^\frown i}}$ is isomorphic to $(T_{t^\frown j}, a_u, \mathcal{A}_u)_{u\in T_{t^\frown j}}$ over $\mathcal{A}_t$. % We write $X^i_t$ for the set of those $j<\kappa^+$ such that $p_{t^{\frown j}}=p_{t^{\frown i}}$.
Notice that the $Y^i_t$'s are intervals of $\kappa^+$.

\begin{defn}
    A QD $(T, a_t, \mathcal{A}_t)_{t\in T}$ is maximal if $(T, a_t, \mathcal{A}_t)_{t\in T}$ is a maximal good labelled tree. In a similar way, we define maximal QD over $\mathcal{A}$ and maximal QD over $(\mathcal{B},\mathcal{A})$.
\end{defn}

\begin{comment}

\begin{defn}
    For $\mathcal{A}\preceq M$ we define QD over $\mathcal{A}$ as in Definition \ref{good_labelled_tree} and Definition \ref{QD_over}, by changing Definition \ref{good_labelled_tree} item 4 to $\mathcal{A}_\emptyset=\mathcal{A}$, i.e. we drop the requirements that involve $\mathcal{B}$. Thus a maximal QD over $\mathcal{A}$ is a QD over $\mathcal{A}$ that doesn't have a proper extension that is a QD over $\mathcal{A}$. A \textit{QD} is a QD over $\mathcal{A}\preceq M$, where $\mathcal{A}$ is primary over the empty set. 
 %   A QD $(T, a_t, \mathcal{A}_t)_{t\in T}$ over $(\mathcal{B}, \mathcal{A})$ is maximal if for all $t\in T$, $ht_T(t)\ge 1$, there is no $a\in M$ such that $a\downarrow_{\mathcal{A}_t}\bigcup\{\mathcal{A}_{t^\frown i}\mid t^\frown i\in T\}$ and $tp(a/\mathcal{A}_{t})\dashv\mathcal{A}_{t^{-}}$ (when $t^-$ exists). 
\end{defn}

\end{comment}

\begin{lemma}\label{5.12}
    $M_\eta$ has a maximal  QD.% over $(\mathcal{B},\mathcal{A})$.
\end{lemma}

\begin{cproof}
    We will construct a maximal QD by recursion. Let us start by constructing $(T^0, \mathcal{A}_t^0)_{t\in T^0}$. Let $T^0=\{\emptyset\}$, $\mathcal{A}_\emptyset^0\subseteq M_\eta$ a primary model over $\emptyset$, and $a_\emptyset^0\in \mathcal{A}_\emptyset^0$. %Recall that $a_{(0,0)}^0=\emptyset$.
    
    Suppose that $\alpha<\kappa^+$ is such that $(T^\alpha, a_t^\alpha, \mathcal{A}_t^\alpha)_{t\in T^\alpha}$ has been constructed. If $(T^\alpha, a_t^\alpha, \mathcal{A}_t^\alpha)_{t\in T^\alpha}$ is a maximal QD, then we are done. 
    Let us take care of the case when $(T^\alpha, a_t^\alpha, \mathcal{A}_t^\alpha)_{t\in T^\alpha}$ is not a maximal QD. 
    Since $(T^\alpha, a_t^\alpha, \mathcal{A}_t^\alpha)_{t\in T^\alpha}$ is not a maximal QD, there is $t\in T^\alpha$ such that $ht_{T^\alpha}(t)>0$ and $(T^\alpha_{t}, a_u^\alpha, \mathcal{A}_u^\alpha)_{u\in T^\alpha_{t}}$ is not a maximal labelled tree over $(\mathcal{A}^\alpha_{t^-},\mathcal{A}^\alpha_t)$, or $ht_{T^\alpha}(t)=0$ and $(T^\alpha_{t}, a_u^\alpha, \mathcal{A}_u^\alpha)_{u\in T^\alpha_{t}}$ is not a maximal labelled tree.
    
    Let us pick $t\in T^\alpha$ such that $ht_{T^\alpha}(t)$ is maximal with such property. Therefore, for all $i<\kappa^+$, if $t^\frown i\in T^\alpha$, then $(T^\alpha_{t^\frown i}, a_u^\alpha, \mathcal{A}_u^\alpha)_{u\in T^\alpha_{t^\frown i}}$ is a maximal labelled tree over $(\mathcal{A}^\alpha_{t},\mathcal{A}^\alpha_{t^\frown i})$. Let $k= \min \{i<\kappa^+\mid t^\frown i\notin T^\alpha\}$. By Definition \ref{QD_over} we will have to deal with three cases.
    \begin{enumerate}
        \item There are $j<k$ and $(T', a_u, \mathcal{A}_u)_{u\in T'}$ such that for all $k>i>j$, $(T^\alpha_{t^\frown i}, a_u^\alpha, \mathcal{A}_u^\alpha)_{u\in T^\alpha_{t^\frown i}}$ is isomorphic to $(T^\alpha_{t^\frown j}, a_u^\alpha, \mathcal{A}_u^\alpha)_{u\in T^\alpha_{t^\frown j}}$  over $\mathcal{A}^\alpha_t$, and $(T', a_u, \mathcal{A}_u)_{u\in T'}$ is isomorphic to $(T^\alpha_{t^\frown j}, a_u^\alpha, \mathcal{A}_u^\alpha)_{u\in T^\alpha_{t^\frown j}}$  over $\mathcal{A}^\alpha_t$, and it is a maximal good labelled tree over $(\mathcal{A}^\alpha_t, \mathcal{A}_{u^*})$ which is also a QD, where $u^*$ is such that $ht_{T'}(u^*)=1$ and $$a_{u^*}\downarrow_{\mathcal{A}_t}\bigcup_{\beta<k}a^\alpha_{t^\frown \beta}.$$ 
        We define $(T^{\alpha+1}_{t^\frown k}, a_u^{\alpha+1}, \mathcal{A}_u^{\alpha+1})_{u\in T^{\alpha+1}_{t^\frown k}}$ as $(T', a_u, \mathcal{A}_u)_{u\in T'}$ and $$(T^{\alpha+1}, a_u^{\alpha+1}, \mathcal{A}_u^{\alpha+1})_{u\in T^{\alpha+1}}$$ as $$(T^\alpha, a_u^\alpha, \mathcal{A}_u^\alpha)_{u\in T^\alpha}\cup (T^{\alpha+1}_{\ge t^\frown k}, a_u^{\alpha+1}, \mathcal{A}_u^{\alpha+1})_{u\in T^{\alpha+1}_{\ge t^\frown k}}.$$
        
        \item There are no $j<k$ and $(T', a_u, \mathcal{A}_u)_{u\in T'}$ as in the previous case, but there are $j<k$ and $a\not \subseteq \mathcal{A}^\alpha_t$ such that for all $k>i>j$, $p_{t^\frown j}= p_{t^\frown i}$ and $tp(a / \mathcal{A}^{\alpha}_{t})=p_{t^\frown j}$ and $$a\downarrow_{\mathcal{A}^\alpha_t} \bigcup_{\beta<k}a^\alpha_{t^\frown \beta}.$$ 
        We define $a_{t^\frown k}^{\alpha+1}=a$, $\mathcal{A}^{\alpha+1}_{t^\frown k}=\mathcal{A}^\alpha_{t}[a]$, and $T^{\alpha+1}=T^\alpha\cup \{t^\frown k\}$.
        
        \item There are no $j<k$ and $(T', a_u, \mathcal{A}_u)_{u\in T'}$ as in the first case, neither $j<k$ and $a\not \subseteq \mathcal{A}^\alpha_t$ as in the second case. 
        \begin{enumerate}
            \item Case $ht_{T^\alpha}(t)>0$. 

            We choose $a\not \subseteq \mathcal{A}^\alpha_t$ such that $tp(a / \mathcal{A}^\alpha_{t})\dashv \mathcal{A}^\alpha_{t^{-}}$, $tp(a / \mathcal{A}^\alpha_{t})$ is non-algebraic, $$a\downarrow_{\mathcal{A}^\alpha_t} \bigcup_{\beta<k}a^\alpha_{t^\frown \beta},$$ and with the least possible $MR(tp(a/\mathcal{A}^\alpha_{t}))$. We define $a_{t^\frown k}^{\alpha+1}=a$, $\mathcal{A}^{\alpha+1}_{t^\frown k}=\mathcal{A}^\alpha_{t}[a]$, and $T^{\alpha+1}=T^\alpha\cup \{t^\frown k\}$.

            \item Case $ht_{T^\alpha}(t)=0$.

            We choose $a\not \subseteq \mathcal{A}^\alpha_t$ such that $tp(a / \mathcal{A}^\alpha_{t})$ is non-algebraic, $$a\downarrow_{\mathcal{A}^\alpha_t} \bigcup_{\beta<k}a^\alpha_{t^\frown \beta},$$ and with the least possible $MR(tp(a/\mathcal{A}^\alpha_{t}))$. We define $a_{t^\frown k}^{\alpha+1}=a$, $\mathcal{A}^{\alpha+1}_{t^\frown k}=\mathcal{A}^\alpha_{t}[a]$, and $T^{\alpha+1}=T^\alpha\cup \{t^\frown k\}$.
        \end{enumerate}
        
    \end{enumerate}

    Suppose  $\alpha$ is a limit ordinal such that for all $\gamma<\alpha$, $(T^\gamma, a_t^\gamma, \mathcal{A}_t^\gamma)_{t\in T^\gamma}$ has been constructed. We define $(T^\alpha, a_t^\alpha, \mathcal{A}_t^\alpha)_{t\in T^\alpha}$ by 
    \begin{itemize}
        \item $T^\alpha=\bigcup_{\gamma<\alpha}T^\gamma$.
        \item For all $t\in T^\alpha$, $a^\alpha_t=a^\gamma_t$ where $\gamma$ is the least ordinal such that $t\in T^\gamma$. 
        \item For all $t\in T^\alpha$, $\mathcal{A}_t^\alpha=\mathcal{A}_t^\gamma$ where $\gamma$ is the least ordinal such that $t\in T^\gamma$. 
    \end{itemize}
\end{cproof}

%\begin{cor}
%    There is a maximal QD $(T,a_t,\mathcal{A}_t)_{t\in T}$ of $M$.
%\end{cor}

Let $(T, a_t, \mathcal{A}_t)_{t\in T}$ be the maximal QD of $M_\eta$ from Lemma \ref{5.12}.

\begin{lemma}\label{Lemma_5-2}
    If $t\in T$, $j\leq i$, $a$, and $b$ are such that $a\models p_{t^\frown i}$, $b\models p_{t^\frown j}$, $\mathcal{A}_t\subseteq A$, $a\downarrow_{\mathcal{A}_t}A$, and $b\not\downarrow_{\mathcal{A}_t}A$. Then $a\downarrow_A b$. In particular $p_{t^\frown i}$ is regular.
\end{lemma}
\begin{cproof}
    Let us assume, for sake of contradiction, that $t\in T$, $j\leq i$, $a$, and $b$ are such that $a\models p_{t^\frown i}$, $b\models p_{t^\frown j}$, $\mathcal{A}_t\subseteq A$, $a\downarrow_{\mathcal{A}_t}A$, and $b\not\downarrow_{\mathcal{A}_t}A$, and $a\not\downarrow_A b$. So there is $d\in A$ such that $a\downarrow_{\mathcal{A}_t}d$, and $b\not\downarrow_{\mathcal{A}_t}d$, and $a\not\downarrow_{\mathcal{A}_t d} b$.
    
    Let us choose $c\in\mathcal{A}_t$ such that $a\downarrow_c \mathcal{A}_t$, $b\downarrow_c\mathcal{A}_t$, $b\not\downarrow_c d$, and there is $\varphi(y,c)\in tp(b,\mathcal{A}_t)$ such that $MR(\varphi(y,c))=MR(tp(b,\mathcal{A}_t))$ and has degree 1, and $a\not\downarrow_c db$. 

    Notice that there is $\psi(x,y,a,c)$ such that $\models \psi(d,b,a,c)$ and $\psi(x,y,a,c)$ forks over $c$. Also, there is $\chi(y,d,c)$ such that $\models\chi(b,d,c)$ and $\chi(b,d,c)$ forks over $c$. Without loss of generality, we may assume that there is $d_\chi(x,c)$ that defines $tp(b,\mathcal{A}_t)\restriction \chi(y,x,c)$. Let us define $$\theta'(x,y,z,c)=\varphi(y,c)\wedge \psi(x,y,z,c)\wedge \neg d_\chi(x,c)\wedge \chi(y,x,c)$$ and $\theta(x,z,c)=\exists y \theta'(x,y,z,c)$. 
    Notice that $\models \theta(d,a,c)$ and $a\downarrow_c d$, therefore there is $d'\in \mathcal{A}_t$ such that $\models \theta(d',a,c)$. Thus there is $b'\in \mathcal{A}_{t^\frown i}$ such that $\models \theta'(d',b',a,c)$, so $\models \psi(d',b',a,c)$. Since $a\not\downarrow_{c}d'b'$, $b'\in \mathcal{A}_{t}[a_{t^\frown i}]\backslash \mathcal{A}_t$ and $a\triangleright_{\mathcal{A}_t} b'$. 
    We conclude that $b'\downarrow_{\mathcal{A}_t}\bigcup_{j<i}a_{t^\frown j}$, $tp(b',\mathcal{A}_t)$ is not algebraic, and $tp(b',\mathcal{A}_t)\dashv\mathcal{A}_{t^-}$.

    Since $\not\models d_\chi (d',c)$, $\chi(y,d',c)\notin tp(b,\mathcal{A}_t)$. %So $MR(tp(b',\mathcal{A}_t))\leq MR(tp(b,\mathcal{A}_t))$.
    On the other hand $\models \chi(b',d',c)$, so $tp(b',\mathcal{A}_t)\neq tp(b,\mathcal{A}_t)$. Since $MR(\varphi(y,c))=MR(tp(b,\mathcal{A}_t))$ and has degree 1, $MR(tp(b',\mathcal{A}_t))< MR(tp(b,\mathcal{A}_t))$ a contradiction with the minimality of $MR(tp(b,\mathcal{A}_t))$.

    \begin{comment}

    Let $c\in \mathcal{A}_{t}$ and $\varphi(x,c)\in p_{t^\frown i}$ such that $MR(\varphi(x,c))=MR(p_{t^\frown i})$, and there is $\psi(x,b,A,c)\in tp(a/\mathcal{A}_{t}Ac)$ such that $\psi$ forks over $c$. Let $\theta(x,b,c)\in tp(a/bc)$ be such that it forks over $c$. We may assume that $c$ is such that $p_{t^\frown i}\restriction\neg \psi(x,y,z,c)$ is definable by some $\varphi^*(y,z,c)$, $p_{t^\frown i}\restriction\neg \theta(x,z,c)$ is definable by some $\theta^*(z,c)$ and denote $$p_{t^\frown i}\restriction \exists y(\psi(x,y,z,c) \wedge \varphi^*(y,z,c)\wedge \theta(x,z,c)\wedge \theta^*(z,c)\wedge \varphi(x,c)\wedge \varphi(y,c))$$ by $\Phi(z,c)$. Since $a\downarrow_{\mathcal{A}_t}A$, $\Phi(A,c)$ holds and $b$ witnesses that $\Phi(z,c)$ is in $tp(a/\mathcal{A}_t A)$. So, we can find $b'\in \mathcal{A}_t$ such that $\Phi(b',c)$ holds. There is $a'\in\mathcal{A}_{t^\frown i}$ such that $\Phi(a',c)$ holds. Then $a'\not\downarrow_c A$, $a'\downarrow_{cA}b$, and $a'\notin\mathcal{A}_t$.

    Therefore $b$ dominates $a'$ over $\mathcal{A}_t$. It follows that $a'$ satisfies Definition \ref{good_labelled_tree} (5) and (6). Since $a'\not\downarrow_c A$, $MR(tp(a'/\mathcal{A}_t))<MR(p_t)$, contradicting Definition \ref{QD_over} (1).
    \end{comment}
\end{cproof}

\begin{cor}
Maximal QD are decompositions. In particular, $M_\eta$ and $M_\xi$ have isomorphic QD's if and only if $M_\eta\cong M_\xi$. 
\end{cor}
\begin{cproof}
    The first claim follows from the \textit{``in particular"} part of Lemma \ref{Lemma_5-2}. For the second claim, clearly $M_\eta\cong M_\xi$ implies that $M_\eta$ and $M_\xi$ have isomorphic QD's. For the other direction, from the proof of Lemma \ref{Lemma_5-2}, it can be seen that $tp(a_t,\mathcal{A}_{t^-})$ are weakly isolated, thus strongly regular (see \cite{MR918762} page 243). Therefore a maximal QD is a K-representation as in Definition 2.1 \cite{MR918762} page 347. From Corollary 2.4 \cite{MR918762} page 348, $M_\eta$ and $M_\xi$ are K-minimal over their QD's. We conclude that if  $M_\eta$ and $M_\xi$ have isomorphic QD's, then $M_\eta\cong M_\xi$.

%and the second is well-know, see e.g. \cite{MR918762} 
\end{cproof}

We will find an $L^*_{\kappa^+\omega_1}$-formula $\Phi_\eta$ such that $M^*_\xi\models \Phi_\eta$ if and only if $M_\xi\cong M_\eta$, i.e. there is a maximal QD of $M_\xi$, which is isomorphic to the maximal one found in Lemma \ref{5.12}.

%For all $t\in T$, let us denote $rk(T_{\ge t})$ by $rk(t)$. 
We prove by induction on $rk(t)$ for $t\in T$, $ht_T(t)>0$, the following:

\textit{There is an $L^*_{\kappa^+\omega_1}$-formula $\Phi^t_\eta(\bar{X},\bar{Y},x)$ such that for all $\xi$, countable $\mathcal{B}\prec \mathcal{A}\prec M_\xi$ and $a\in\mathcal{A}$; $M_\xi\models \Phi^t_\eta(\mathcal{A},\mathcal{B},a)$ if and only if there is a maximal QD, $(T',a'_u,\mathcal{A}'_u)_{u\in T'}$, over $(\mathcal{B},\mathcal{A})$ such that $a'_{u^*}=a$ where $ht_{T'}(u^*)=1$ and $(T',a'_u,\mathcal{A}'_u)_{u\in T'}$ is isomorphic with $(T_t,a_u,\mathcal{A}_u)_{u\in T_t}$ over $\mathcal{A}_t$.}

%There is $\Phi^t_\eta(\bar{X},\bar{Y})$ such that for all $\mathcal{B}\preceq\mathcal{A}\preceq M_\xi$ countable, $M_\xi\models \Phi^t_\eta(\mathcal{A},\mathcal{B})$ if and only if there is a maximal QD over $(\mathcal{A},\mathcal{B})$ of $M_\xi$ such that it is isomorphic with $(T_t, a_u,\mathcal{A}_u)_{u\in T_t}$.}

{\bf Case $rk(t)=0$, i.e. a leaf.} Let $\Phi_\eta^t(\bar{X},\bar{Y}, y)$ be the conjunction of all first order formulas $\varphi(\bar{X},\bar{Y}, y)$ such that $M_\eta\models \varphi(\mathcal{A}_t,\mathcal{A}_{t^-},a_t)$ together with a formula that says that there is no $a$ such that $tp(a,\bar{X})$ is non-algebraic and orthogonal to $\bar{Y}$.

%Let $\Phi_\eta^t(\bar{X},\bar{Y}, y)$ be the formula that says $M_\xi$ is a model of $\mathcal{T}$ and there is no $a$ such that $tp(a,\mathcal{A} _t)$ is non-algebraic and orthogonal to $\mathcal{A}_{t^-}$, and describes the isomorphism types of $\mathcal{A}_t$ and $\mathcal{A}_{t^-}$%a_t$ over $\emptyset$ 
%(i.e. there is $f$ an isomorphism between $\bar{X}$ and $\mathcal{A}_t$ such that $f[\bar{Y}]=\mathcal{A}_{t^-}$). 

{\bf Case $rk(t)>0$ and $ht_T(t)>0$.} Let $\langle Y_\alpha\rangle_{\alpha<\gamma}$ be an increasing enumeration of the set $\{Y^i_t\mid i<\kappa^+\}$, i.e. if $i<j$, $i\in Y_\alpha$ and $j\in Y_\beta$, then $\alpha\leq \beta$. Notice that there are $\kappa$-many finite tuples in $M_\eta$, thus $\gamma<\kappa^+$.
%For all $\alpha<\gamma$, let us denote by $\gamma_\alpha^0$ the least element of $Y_\alpha$.
By the induction assumption, for all $\alpha<\gamma$ and $i\in Y_\alpha$, let $\Phi_\eta^\alpha=\Phi_\eta^{t^\frown i}$ and notice that since for all $i,j\in Y_\alpha$, $(T_{t^\frown i}, a_u, \mathcal{A}_u)_{u\in T_{t^\frown i}}$ and $(T_{t^\frown j}, a_u, \mathcal{A}_u)_{u\in T_{t^\frown j}}$ are isomorphic over $\mathcal{A}_t$, we can construct the formulas $\Phi_\eta^{t^\frown i}$ so that it does not depends on $i\in Y_\alpha$.

We order $\gamma\times \kappa$ lexicographically, i.e. $(\alpha_0,i_0)<(\alpha_1,i_1)$ if either $\alpha_0<\alpha_1$, or $\alpha_0=\alpha_1$ and $i_0<i_1$.
For all $\alpha<\gamma$ and $i<\kappa$ we define $\Psi^i_\alpha (\bar{X},x)$:

\textit{$x$ is the $<^n$-least tuple such that: 
\begin{itemize}
    \item $\exists \bar{Z}\Phi^\alpha_\eta(\bar Z,\bar X, x)$ holds and so %there is a countable $A$ such that $A\mathcal{A}_t x$ satisfies $\Phi^\alpha_\eta$, 
    \item for all $(\alpha_0,i_0)<(\alpha_1,i_1)<\cdots <(\alpha_k,i_k)<(\alpha,i)$, if $x_0, \ldots, x_k$ are such that for all $j<k$, $x_j$ satisfies $\Psi_{\alpha_j}^{i_j}$, then $x\downarrow_{\bar X}x_0, \ldots, x_k$.
\end{itemize}}

Notice that not necessarily each $\Psi_\alpha^i$'s has a realization and in case it has, it is unique. For all $\alpha<\gamma$, let us denote by $\gamma_\alpha<\kappa$ the least $i$ such that $\Psi_\alpha^i$'s has no realization (in case it exists). Notice that if $\gamma_\alpha\leq j$, then $\Psi_\alpha^j$'s has no realization.
%Let $Y'_\alpha$ be the set of all $i<\kappa$ such that $\Psi_\alpha^i$ has a realization. 

For all $\alpha<\gamma$, let us define $\Psi_\alpha$ as follows:
\begin{itemize}
    \item $\bigwedge_{i<\kappa}\exists x\Psi^i_\alpha$ if $\gamma_\alpha$ does not exist.
    \item $\bigwedge_{i<|\gamma_\alpha|}\exists x\Psi^i_\alpha\ \wedge\ \bigvee_{i<|\gamma_\alpha|^+}\neg \exists x\Psi^i_\alpha$ if $\gamma_\alpha$ exists.
\end{itemize}

Let $\Phi^t_\eta(\bar X,\bar Y, y)$ be the formula that says:
\begin{itemize}
    \item $\bigwedge_{\alpha<\gamma}\Psi_\alpha$,
%then $\Phi^t_\eta$ describes the type of $\mathcal{A}_t\mathcal{A}_{t^-}a_t$;
\item conjunction of all first order formulas $\varphi(\bar X,\bar Y, y)$ such that $\models \varphi(\mathcal{A}_t,\mathcal{A}_{t^-},a_t)$,
\item there is no tuple $a$ such that $tp(a/\bar X)$ is non-algebraic, $tp(a/\bar X)\dashv \bar Y$ and if $x_0, \ldots, x_n$ satisfies the formulas $\Psi^i_\alpha$, then $a\downarrow_{\bar X}x_0, \ldots, x_k$.
\end{itemize}
%\begin{enumerate}
%    \item If $|Y_\alpha|<\kappa$: its says that there is $i<\kappa$ such that $|i|=|Y\alpha|$, $Y^i_\alpha$ does not have a realization but for all $j<i$, $\psi_\alpha^k$ has it.
%    \item $|Y_\alpha|=\kappa$: it says that $\Psi^i_\alpha$ has a realization for all $i<\kappa$.
%\end{enumerate}

%{\bf Case $rk(t)>0$ and $ht_T(t)=0$.} Similar as in the previous case, taking care of $t^-$. Notice that $t^-$ does not exists and $a_t=\emptyset$, so $\mathcal{A}_t$ is a primary model over the empty set.

Notice that for $t$ such that $ht_T(t)=0$ we do not need to talk about $\bar Y$ nor $y$.
For $t$ such that $ht_T(t)=0$, we define $\gamma$ as above and for all $\alpha<\gamma$, we define $\Psi_\alpha$ as above. Thus $\Phi^t_\eta(\bar X)=\bigwedge_{\alpha<\gamma}\Psi_\alpha$. %Notice that since $ht_T(t)=0$ we do not need to talk about $\bar Y$ nor $x$.

Finally $\Phi_\eta=\exists\bar X \Phi^t_\eta(\bar X)$ for $t$ the root of $T$.%says that there is $\mathcal{A}\preceq M_\xi$ primary over $\emptyset$ such that $\Phi^\emptyset_\eta(\mathcal{A})$ holds. 

Now we can proof Theorem \ref{orbit_thm} by showing that $\Phi_\eta$ is the desire formula.

\begin{cproof}[Proof Theorem \ref{orbit_thm}]
$\Leftarrow$) It is clear that if $M^*_\xi\models \Phi_\eta$ and $M^*_{\xi'}\models \Phi_\eta$, then $M_\xi$ and $M_{\xi'}$ have isomorphic QD's and thus $M_\xi\cong M_{\xi'}$.

%It is clear that for all $t$ such that $t^-$ exists, if $M_\xi\models \Phi^t_\eta(\mathcal{A},\mathcal{B})$, then the maximal QD over $(\mathcal{A},\mathcal{B})$ is isomorphic to $(T_t, a_u,\mathcal{A}_u)_{u\in T_t}$. Thus $M_\eta\cong M_\xi$. 

$\Rightarrow$) We will start by showing the following
\begin{itemize}
    \item $M_\eta^*\models \Phi_\eta$,
    \item if $M_\eta^*\models \Phi_\eta$ and $M_\xi\cong M_\eta$, then $M_\xi^*\models\Phi_\eta$. 
\end{itemize}

For this we will show first that $M_\eta^*\models\Phi_\eta^t(\mathcal{A}_t)$ where $t$ is the root of $T$. For this we will show first that for all $t\in T$ such that $ht_T(t)>0$,  $M_\eta^*\models\Phi_\eta^t(\mathcal{A}_t,\mathcal{A}_{t^-},a_t)$ holds. We will proceed by induction over $rk(t)$.

{\bf Case $rk(t)=0$}. $M_\eta^*\models\Phi_\eta^t(\mathcal{A}_t,\mathcal{A}_{t^-},a_t)$ follows from the construction of $\Phi_\eta^t(\mathcal{A}_t,\mathcal{A}_{t^-},a_t)$ and the fact that $(T_t,a_u,\mathcal{A}_u)_{u\in T_t}$ is a maximal QD of $M_\eta$.

{\bf Case $rk(t)>0$}. As in the construction of $\Phi^t_\eta$, let $Y_\alpha$, $\alpha<\gamma$, be the enumeration of the sets $Y_t^i$. For all $\alpha<\gamma$, let $Y'_\alpha$  be the set of those $i<\kappa$ such that $\psi_\alpha^i(\mathcal{A}_t,x)$ has a realization.
We call this realization as $b_\alpha^i$. We will prove that for all $\alpha<\gamma$, the following hold:
\begin{enumerate}
    \item $|Y'_\alpha|=|Y_\alpha|$;
    \item for all $i\in Y'_\alpha$, $b_\alpha^i\not\downarrow_{\mathcal{A}_t}\bigcup_{\beta\leq \alpha}\bigcup_{j\in Y_\beta}a_{t^\frown j}$;
    \item for all $i\in Y_\alpha$, $a_{t^\frown i}\not\downarrow_{\mathcal{A}_t}\bigcup_{\beta\leq \alpha}\bigcup_{j\in Y_\beta}b_\beta^j$.    
\end{enumerate}

Let $\alpha$ be such that for all $\beta<\alpha$, (1), (2), and (3) hold.
Let $W$ be the set of all $a\in M_\eta^*$ such that it satisfies $\Phi^\alpha_\eta$ and $a\downarrow_{\mathcal{A}_t}\bigcup_{\beta< \alpha}\bigcup_{j\in Y_\beta}a_{t^\frown j}$, and $W'$ be the set of all $a\in M_\eta^*$ such that it satisfies $\Phi^\alpha_\eta$ and $a\downarrow_{\mathcal{A}_t}\bigcup_{\beta< \alpha}\bigcup_{j\in Y_\beta}b_\beta^j$.
Notice that $\{a_t^i\mid i\in Y_\alpha\}$ is a basis of $W$ and $\{b_\alpha^i\mid i\in Y'_\alpha\}$ is a basis of $W'$. So it is enough to show that $W=W'$.
We will show that $W\subseteq W'$, the other direction is is similar.

   We will proceed by contradiction, to show that for all $a\in W$, $a\downarrow_{\mathcal{A}_t}\bigcup_{\beta< \alpha}\bigcup_{j\in Y_\beta}b_\beta^j$ (i.e. $W\subseteq W'$).
Let us assume, for sake of contradiction, that there are $a\in W$, $n<\omega$, and $b_0,\ldots,b_n\in \{b_\beta^i\mid i\in Y'_\beta, \beta<\alpha\}$ such that $a\not\downarrow_{\mathcal{A}_t}b_0\cdots b_n$. Since $a\downarrow_{\mathcal{A}_t}\bigcup_{\beta< \alpha}\bigcup_{j\in Y_\beta}a_{t^\frown j}$, 
\begin{itemize}
    \item [($\ast$)] \begin{center}
        $a\not\downarrow_{\mathcal{A}_t\bigcup_{\beta< \alpha}\bigcup_{j\in Y_\beta}a_{t^\frown j}}b_0\cdots b_n.$
    \end{center}
\end{itemize}
We choose $n$ so that it is the least for which ($*$) holds.
Let us denote $\bigcup_{\beta< \alpha}\bigcup_{j\in Y_\beta}a_{t^\frown j}$ by $Z^*$. 
It is enough to prove the claim below, it contradicts ($\ast$). 
\begin{claim}
    $a\downarrow_{\mathcal{A}_tZ^*}b_0 \cdots b_n$.
\end{claim}
\begin{cproof}
Since by the induction assumption $b_n\not\downarrow_{\mathcal{A}_t}Z^*$ and by the choice of $n$, $a\downarrow_{\mathcal{A}_t Z^*}b_0\cdots b_{n-1}$, by Lemma \ref{Lemma_5-2}, $tp(a,\mathcal{A}_tZ^* b_0 \cdots b_{n-1})$ and $tp(b_n,\mathcal{A}_tZ^* b_0 \cdots b_{n-1})$ are orthogonal. Since $$b_n\downarrow_{\mathcal{A}_t\bigcup_{\beta< \alpha}\bigcup_{j\in Y_\beta}a_{t^\frown j}}b_0\cdots b_{n-1},$$ we have $a\downarrow_{\mathcal{A}_tZ^*b_0\cdots b_{n-1}}b_n$.
Thus by transitivity, $a\downarrow_{\mathcal{A}_tZ^*}b_0 \cdots b_n$.
%    Let us proceed by induction. Suppose $m< n$ is such that either $0<m$ and $a\downarrow_{\mathcal{A}_t}Z^* b_0 \cdots b_{m-1}$, or $m=0$ and $a\downarrow_{\mathcal{A}_t}Z^*$. By the induction hypothesis item (2), for all $i\leq n$, $b_i\not\downarrow_{\mathcal{A}_t}Z^*$. Thus $b_m\not\downarrow_{\mathcal{A}_t}Z^* b_0 \cdots b_{m-1}$. By Lemma \ref{Lemma_5-2}, $tp(a,\mathcal{A}_tZ^* b_0 \cdots b_{m-1})$ and $tp(b_m,\mathcal{A}_tZ^* b_0 \cdots b_{m-1})$ are orthogonal.
    
%  Thus $a\downarrow_{\mathcal{A}_t Z^* b_0 \cdots b_{m-1}}b_m$. By transitivity and the induction hypothesis of the claim, $a\downarrow_{\mathcal{A}_t Z^*}b_0 \cdots b_m$. Notice that since $a\downarrow_{\mathcal{A}_t}\bigcup_{\beta< \alpha}\bigcup_{j\in Y_\beta}a_{t^\frown j}$, by transitivity $a\downarrow_{\mathcal{A}_t}Z^* b_0 \cdots b_{m}$.
     
%     On the other hand, by the induction hypothesis item (2), for all $i\leq n$, $b_\alpha^i\not\downarrow_{\mathcal{A}_t}\bigcup_{\beta\leq \alpha}\bigcup_{j\in Y_\beta}a_{t^\frown j}$. By Lemma \ref{Lemma_5-2}, $tp(b,\mathcal{A}_t \cup \bigcup_{\beta\leq \alpha}\bigcup_{j\in Y_\beta}a_{t^\frown j})$ is orthogonal to $tp(b_\alpha^i,\mathcal{A}_t \cup \bigcup_{\beta\leq \alpha}\bigcup_{j\in Y_\beta}a_{t^\frown j})$ for all $i\leq n$. Thus $b\downarrow_{\mathcal{A}_t\bigcup_{\beta\leq \alpha}\bigcup_{j\in Y_\beta}a_{t^\frown j}}b_0,\ldots,b_n$, a contradiction.
 \end{cproof}

We conclude that $|Y'_\alpha|=|Y_\alpha|$, and items (2) and (3) follow from the fact that $\{a_t^i\mid i\in Y_\alpha\}$ and $\{b_\alpha^i\mid i\in Y'_\alpha\}$ are basis of $W$ and $W'$, respectively. We conclude that $M_\eta^*\models\Phi_\eta^t(\mathcal{A}_t,\mathcal{A}_{t^-},a_t)$.
Following the same argument, we show that $M_\eta^*\models \Phi_\eta^t(\mathcal{A}_t)$ holds for $t$ the root of $T$. So  $M_\eta^*\models \Phi_\eta$.

Notice that above, the well-orders $<^n$ were used only to express the dimensions in the required infinitary logic. Neither of these dimensions nor the formulas that express them depend on the choice of the well-orders, as long as the order type of the orders is $\kappa$. Thus $M^*_\eta\models\Phi_\eta$ is independent of the choice of the well-orders in $M^*_\eta$.

%Notice that the well-order was not used in the proof. Thus, $M_\eta^*\models \Phi_\eta$ is independent of the well-order. 
To prove the second item, let $\xi$ be such that $M_\xi\cong M_\eta$ and $\Pi: M_\eta \rightarrow M_\xi$ be an isomorphism. 
Let us define the following orders, $<^{n}_\eta$ for $n<\omega$, on $M_\eta$: $$x<^{n}_\eta y\ \Leftrightarrow\ \Pi(x)<^n \Pi(y)$$
where $<^n$ is the well-ordering of $\kappa^n$ of order type $\kappa$ in $M^*_\xi$. 
Let $M_\eta^{**}$ be the expansion of $M_\eta$ to $L^*$ with the orders $<^{n}_\eta$ as the interpretation of the well-orderings. Thus $\Pi$ is an isomorphism between $M_\xi^*$ and $M_\eta^{**}$. %Since the well-order of $M^*_\eta$ was not used to prove that $M_\eta^*\models \Phi_\eta$, 
%We notice that $\Phi_\eta$ depends on the interpretation of the well-orderings only up to what the cardinalities of the ordinals $\gamma_\alpha$ are. Above we showed, that this does not depend on the interpretations (we have not assumed anything on the original interpretations of the well-orderings, other than the order type is $\kappa$).
Since $\Phi_\eta$ is independent of the choice of the well-orders (as far as the order type of the orders is $\kappa$), we conclude that $M_\eta^{**}\models \Phi_\eta$. So $M_\xi^{*}\models \Phi_\eta$, as wanted.
\end{cproof}

\begin{cor}\label{cor515}
    Let $\mathcal{T}$ be a countable $\omega$-stable NDOP shallow $L$-theory. Suppose $\kappa$ is a regular cardinal such that $\kappa^\omega=\kappa$. If the number of isomorphism classes is at most $\kappa$, then $\cong_\mathcal{T}$ is $\kappa$-Borel.
\end{cor}

Notice that an $\omega$-stable NDOP shallow $L$-theory is a classifiable shallow theory. By Shelah's Main Gap Theorem, the number of isomorphism classes of a classifiable shallow theory is less than $\beth_{\omega_1}(|\alpha|)$, where $\kappa=\aleph_\alpha$. Thus if $\kappa=\aleph_\alpha$ is a regular cardinal such that $\kappa^\omega=\kappa$ and $\beth_{\omega_1}(|\alpha|)\leq \kappa$, and $\mathcal{T}$ is an $\omega$-stable NDOP shallow $L$-theory, then $\cong_\mathcal{T}$ is $\kappa$-Borel. 

\begin{remark}
    Let $\mathcal{T}$ be a countable stable theory, $\kappa$ regular and $\kappa^\omega=\kappa$. Then the set of all $\eta$ such that $M_\eta\models \mathcal{T}$ is strongly $\omega$-saturated is $\kappa$-Borel. 
\begin{cproof}
    It is enough to find a $L_{\kappa^+\omega_1}$-sentence $\varphi$ such that for all $\eta$, $M_\eta\models \mathcal{T}$ is strongly $\omega$-saturated if and only if $M_\eta\models\varphi$.

    We construct $\varphi$ as follows. Let $\langle P^n_i\mid i<2^\omega\rangle$ be the list of all complete $n$-types $p(x)$ over the empty set. For all $i<2^\omega$ and $0<m<\omega$, let $\langle\theta_j^{im}\mid j<\omega\rangle$ be the list of all formulas $\theta(u,v,x)$ such that if $a\models P^n_i$, then $\theta(u,v,a)$ defines a finite equivalence relation on $M^m$. For all $i<2^\omega$ and $0<m<\omega$, let $\langle q^{im}_k\mid k<2^\omega\rangle$ the list of all complete types $q(x,u_j)_{j<\omega}$ over the empty set such that $q^{im}_k\vdash P^n_i$ and if $(a,a_j)_{j<\omega}\models q^{im}_k$, then there is $b$ such that $\models \theta^{im}_j(a_j,b,a)$ for all $j<\omega$. For all $i<2^\omega$ and $0<m<\omega$, let $$\Psi^i_m=\forall u_0\forall u_1\cdots (\bigwedge_{k<2^\omega}(\bigwedge q_k^{im}\rightarrow \exists v\bigwedge_{j<\omega}\theta_j^{im}(u_j,v,x)))).$$
    Now, for all $n<\omega$ and $0<m<\omega$, let $\varphi^n_m=\forall x \bigwedge_{i<2^\omega}(\bigwedge P^n_i\rightarrow \Psi^i_m)$. Finally, let $$\varphi=\bigwedge \mathcal{T}\ \wedge\bigwedge_{n<\omega,\ 0<m<\omega} \varphi^n_m.$$ 
    Since for every $M\models \mathcal{T}$ and for every $a\in M^n$, every equivalence class of every equivalence relation definable over $a$ has a representative in $M$; $M\models \varphi$ implies that $M$ is strongly $\omega$-saturated. The other direction is clear.
\end{cproof}
\end{remark}

\begin{thm}
    Let $\mathcal{T}$ be a countable superstable NDOP shallow $L$-theory. Suppose $\kappa$ is a regular cardinal such that $\kappa^{(2^\omega)}=\kappa$. If $M_\xi\models \mathcal{T}$ is strongly $\omega$-saturated, then $\Orb(M_\xi)$ is $\kappa$-Borel. 
\end{thm}

The proof of this theorem is essentially the same as that of Theorem \ref{orbit_thm}. Only obvious changes are needed, e.g. Morley rank is replaced by U-rank.
%The proof follows as in Theorem \ref{orbit_thm}, but instead of using Morley's rank in Lemma \ref{Lemma_5-2} use U-rank.

\section{Open questions}
\begin{enumerate}
    \item Suppose $2^\omega=2^{\omega_1}=\omega_2$. What can be said about $\omega_1$-Borelness of (orbits of) models of cardinality $\aleph_1$? 
    \item  Suppose $\kappa^\omega=\kappa=\cf(\kappa)$. Does it follow that a countable first order theory $T$
is classifiable and shallow if and only if $\cong_T$ is $\kappa$-Borel on $\kappa^\kappa$?
    \item Are the following inclusions proper:
    \begin{itemize}
        \item $\kappa$-Borel $\subseteq \kappa$-$\Delta_1^1$.
        \item $\kappa$-$\Delta_1^1\subseteq (\kappa,\kappa^{<\kappa})$-\borelstar.
    \end{itemize}
    \item Does $(\kappa,\kappa^{<\kappa})$-\borelstar$=\kappa$-$\Sigma_1^1$ imply $\kappa^{<\kappa}=\kappa$.
    \item Do Generalized Baire Spaces have a separation property, under the assumption $\kappa^{<\kappa}>\kappa$? 
\end{enumerate}

\bibliographystyle{plain}
\bibliography{HyMoVa}

\end{document}